\documentclass[12pt]{amsart}

\usepackage{amsmath, amsthm, amsfonts, amssymb}
\usepackage{algorithm}
\usepackage{algorithmic}
\usepackage{xcolor}
\usepackage{hyperref}
\usepackage{soul}
\usepackage{graphicx}
\usepackage{mathrsfs}
\usepackage{rotating,comment}

\usepackage{fullpage,tikz-cd}
\usepackage{tikz,calc}
\usepackage{subcaption}
\usetikzlibrary{shapes.geometric,arrows,positioning}

\usepackage[foot]{amsaddr}

\renewcommand{\Pr}{\mathbb{P}}
\newcommand{\Exp}{\mathbb{E}}

\newcommand{\Poi}{\text{Poi}}
\newcommand{\PoiP}{\mathscr{P}_{\text{poi}}}
\newcommand{\PoiF}{\mathscr{P}_{n, d}^{C}}
\newcommand{\PoiD}{\mathscr{P}_{n, d}^{D}}
\newcommand{\PoiM}{\mathscr{P}_{n, d}^{M}}
\newcommand{\PoiFt}{\mathscr{P}_{n, d}^{C^\prime}}
\newcommand{\PoiMt}{\mathscr{P}_{n, d}^{M^\prime}}
\newcommand{\PoiDt}{\mathscr{P}_{n, d}^{D^\prime}}

\newcommand{\red}[1]{\textcolor{red}{#1}}
\newcommand{\dy}[1]{\textcolor{magenta}{#1}}
\newcommand{\gt}[1]{\textcolor{blue}{#1}}
\newcommand{\remove}[1]{}

\renewcommand{\red}[1]{#1}
\renewcommand{\dy}[1]{#1}
\renewcommand{\gt}[1]{#1}

\newtheorem{theorem}{Theorem}[section]
\newtheorem{corollary}[theorem]{Corollary}
\newtheorem{lemma}[theorem]{Lemma}
\newtheorem{definition}[theorem]{Definition}
\newtheorem{remark}[theorem]{Remark}
\newtheorem{prop}[theorem]{Proposition}

\numberwithin{equation}{section}
\numberwithin{theorem}{section}

\newcommand{\be}{\begin{equation}}
\newcommand{\ee}{\end{equation}}
\normalsize	

\newcommand{\Real}{\mathbb R}

\newcommand{\To}{\longrightarrow}

\newcommand{\A}{\mathcal{A}}

\newcommand{\Dgm}{\mathrm{Dgm}}
\newcommand{\discrete}{\mathbb{N}}

\newcommand{\bdr}{\partial}
\newcommand{\cobdr}{\delta}

\newcommand{\Hg}{\mathrm{H}}
\newcommand{\im}{\mathrm{im}\;}

\newcommand{\md}{\mathrm{d}}
\newcommand{\1}{\mathbf{1}}
\newcommand{\df}[1]{ \textnormal{d}#1}

\def\M{\mathcal{M}}
\def\sC{\mathscr{C}}
\def\F{\mathcal{F}}
\def\sF{\mathscr{F}}
\def\U{\mathcal{U}}
\def\cL{\mathcal{L}}
\def\Lp{\mathcal{L}^\prime}
\def\Dp{D^\prime}
\def\Mp{M^\prime}
\def\K{\mathcal{K}}
\def\IR{\mathcal{I}(\bR)}
\def\S{\mathcal{S}}
\def\sP{\mathscr{P}}

\def\G{\mathcal{G}}
\def\sI{\mathscr{I}}
\def\bH{\mathbb{H}}
\def\bF{\mathbb{F}}
\def\bR{\mathbb{R}}
\def\bZ{\mathbb{Z}}
\def\ccR{C_{c}^{+}(\bR)}
\def\borR{\mathcal{B}(\mathbb{R})}
\def\EP{\mathbb{E}}
\def\phip{\phi^\prime}
\def\rank{\mathrm{rank}}
\def\supp{\mathrm{supp}}

\begin{document}

\title[Random $d-$complexes]{Randomly weighted $d-$complexes: Minimal spanning acycles and Persistence diagrams}%
\author{Primoz Skraba}
\address[PS]{School of Mathematical Sciences, Queen Mary University of London, UK \& Jozef Stefan Institute, Ljubljana, Slovenia. \textnormal{Research supported  by ARRS project N1-0058.}}
%
\email{primoz.skraba@ijs.si}%

\author{Gugan Thoppe$^*$}
\thanks{$^*$Corresponding Author. Postal Address: Dept. of Computer Science and Automation, Indian Institute of Science, Bengaluru, Karnataka 560012, India; Tel.: +91 996-906-6600; Fax:+91-80-23602911}
\address[GT]{Dept. of Computer Science and Automation, Indian Institute of Science, Bengaluru, India; Past: Faculty of Electrical Engineering, Technion-Israel Institute of Technology, Haifa, Israel. \textnormal{Research supported by URSAT, ERC Grant 320422.}}%
\email{gugan.thoppe@gmail.com}

\author{D. Yogeshwaran}
\address[DY]{Statistics and Mathematics Unit, Indian Statistical Institute, Bengaluru, India. \textnormal{Research supported by DST-INSPIRE faculty award.}}
\email{d.yogesh@isibang.ac.in}


\maketitle
\begin{abstract}
\gt{A weighted $d-$complex is a simplicial complex of dimension $d$ in which each face is assigned a real-valued weight. We derive three key results here concerning persistence diagrams and minimal spanning acycles (MSAs) of such complexes. First, we establish an equivalence between the MSA face-weights 
and \emph{death times} in the persistence diagram. 
Next, we show a novel stability result for the MSA face-weights which, due to our first result, also  holds true for the death and birth times, separately. 
Our final result concerns a perturbation of a mean-field model of randomly weighted $d-$complexes. The $d-$face weights here are perturbation of some i.i.d. distribution while all the lower-dimensional faces have a weight of $0$. If the perturbations decay sufficiently quickly, we show that suitably scaled extremal nearest face-weights, face-weights of the $d-$MSA, and the associated death times converge to an inhomogeneous Poisson point process. This result completely characterizes the extremal points of persistence diagrams and MSAs. The point process convergence and the asymptotic equivalence of three point processes are new for any weighted random complex model, including even the non-perturbed case.} Lastly, as a consequence of our stability result, we show that Frieze's $\zeta(3)$ limit \cite{Frieze85} for random minimal spanning trees and the recent extension to random MSAs by Hino and Kanazawa \cite{Hino2018} also hold in suitable noisy settings.
\end{abstract}

%

\section{Introduction}
Broadly, there are two parts to this paper. The first part concerns weighted simplicial complexes. This study significantly deepens the understanding of the relationship of minimal spanning acycles in such complexes to associated persistence diagrams and also to what we refer to as ``nearest face'' distances. The second part looks at a specific ``mean-field" model of complexes with random weights and, in parallel, also considers its perturbations. We refer to these complexes as randomly weighted $d-$complexes or, simply, weighted random complexes. Our results  completely characterize the extremal behaviour of the persistence diagram and the nearest face distances associated with such complexes and then, using the above relationships, also of their minimal spanning acycles.

The motivation for this work comes from the much more studied scenario of weighted graphs, the $1-$dimensional analogue of weighted simplicial complexes, and their random counterparts. A weighted graph can either be viewed in its entirety or as a process wherein it is sequentially built by adding edges in an order dictated by their weights. Taking the former perspective, a minimal spanning acycle corresponds to a minimal spanning tree, while the nearest face distances are basically the nearest neighbour distances. The other viewpoint helps interpret the persistence diagram associated with a graph; informally, it is a record of the ``death times", i.e., the weight values of those edges that connect \dy{a priori disjoint} components.

The fact that connectivity and nearest neighbour distances are intertwined 
can be seen from the earliest work itself on random graphs by Erd\H{o}s and R\'{e}nyi \cite{Erdos59}. \gt{In fact, three years earlier,  Kruskal had proposed his algorithm  for constructing a minimal spanning tree. The edge weights of this tree are precisely the times at which components get connected in the process type description of the weighted graph. Hence, that work can be viewed as the first to implicitly exploit the connections between minimal spanning trees and persistence diagrams.} \dy{\gt{Indeed,} the notion of persistence diagrams did not exist then, but one interpretation of Kruskal's algorithm is via persistence diagrams and this \gt{relation} is made more clear in this paper.} 
This implicit relationship was also used later in the seminal work of Frieze \cite{Frieze85}. On the other hand, connections between the largest nearest neighbour distances and the longest edges of a minimal spanning tree on randomly weighted graphs have played a key role in \cite{Henze82, Steele86, Appel02, Penrose97, Hsing05}. In \cite{Penrose97}, it was shown that the extremal nearest neighbour distances coincided with that of extremal edge-weights in a Euclidean random minimal spanning tree. Such a result is crucial to understanding the connectivity threshold for random geometric graphs (see \cite[Chapter 13]{Penrose03}). More complete accounts of such connections can be found in \cite{Bollobas01, Hofstad16, Janson00, Penrose03, Frieze16}.

Recent applications in topological data analysis have motivated the extension of the above results to random complexes. While higher dimensional analogues of connectivity thresholds \gt{have} already been studied \cite{linial2006homological, meshulam2009homological, kahle2014sharp}, this work generalizes some of these later results to the level of persistence diagrams and minimal spanning acycles. Before delving into the background and details of our results, we summarize our main contributions. \gt{Note that a weighted $d-$complex is a simplicial complex with dimension $d$ in which each face is assigned a \dy{real-valued weight.} Throughout, we will assume that this weight function is monotone, i.e., the weight of any face is always larger than that of its sub-faces. Such a weighted complex can also be viewed as a process wherein one adds faces in the order dictated by their weights. Because the weight function is monotone, any  intermediate construction is also a simplicial complex. With this dual perspective, one can infer properties  about minimal spanning acycles from the death and birth times in the persistence diagram and vice versa.} 

The section numbers in brackets below indicate where one can find a detailed description of the corresponding contribution.\\

\noindent \textbf{Key Contributions}: \begin{enumerate}
%
%
\item  \gt{We first provide a \dy{simplicial analogue} of Kruskal's algorithm that can be used for finding minimal spanning acycles. Comparing this  algorithm with the incremental algorithm used to build a  persistence diagram, we establish an equivalence between the face-weights of minimal spanning acycles and the death times; a similar result also holds true for the birth times. This result significantly enhances the connection between minimal spanning acycles and persistence diagrams. In fact, one of the theorems in Hiraoka and Shirai \cite[Theorem 1.1]{hiraoka2015minimum} now becomes a simple corollary of our result. (Section~\ref{ssec:pd-msa}).} 

\item \gt{Next, we establish a new stability result for minimal spanning acycles. Because of the equivalence above, this result then automatically applies to the death and birth times as well. Unlike existing stability results for persistence diagrams which concern the multiset of birth-death pairs, our result specifically relates the changes in the set of deaths and, separately, in the set of births to the changes in the face weights. We believe this result can play a crucial role in proving results for randomly weighted complexes with certain dependencies between the different face weights. (Section~\ref{ssec:stability})} 

\item \gt{Our final key result concerns randomly weighted $d-$complexes and suitable noisy perturbations of them, including those with dependencies. If the perturbations decay sufficiently fast, we show that appropriately} scaled versions of the following three point processes: (a) nearest face distances, (b) death times in the persistence diagram, and (c) face-weights of the minimum spanning acycle -- converge weakly to the same Poisson point process in vague topology. \gt{Derivation of (a) and (b) involves use of the method of factorial moments, cohomology theory, and our stability result. On the other hand, going from (b) to (c) is a simple application of our first result. However, unlike in our second contribution, notice that this time we exploit the equivalence in the other direction, i.e., we go from a result on death times to a result on face weights of the minimal spanning acycle.} An important paradigm in topological data analysis is that extremal points of a persistence diagram encode meaningful topological information about the underlying structure. Viewed in this light, our result completely characterizes the extremal points of the persistence diagram of these weighted random complexes; this is new even in the non-noisy scenario. (Section \ref{ssec:rc}).

\item We conclude by providing another application of our stability result. Namely, the lifetime sum of persistence diagrams converge for randomly weighted $d-$complexes with noisy weights; this generalizes the $\zeta(3)$-limit for random minimal spanning trees by Frieze \cite{Frieze85} and the recent extension to random spanning acycles by Hino and Kanazawa \cite{Hino2018}. (Section \ref{ssec:rc}).

\end{enumerate}

\paragraph{\sc Organisation of Paper:} The rest of this section quickly introduces simplicial complexes, minimal spanning acycles, persistent homology and then provides more precise statements of some of our main results as well as place them in context. The next section - Section \ref{sec:prelim} - gives in detail the necessary topological (Section \ref{sec:topology}) and probabilistic preliminaries (Section \ref{sec:probability})\footnote{In our effort to make this paper reasonably self-contained as well as accessible to the trio of probabilists, combinatorialists and topologists, we have erred on the side of including too much detail rather than terseness.}. Section \ref{sec:MSA} is exclusively devoted to studying various properties of minimal spanning acycles, algorithms to  find them,  their connection to persistence diagrams, and our crucial stability result. \gt{Finally, in Section~\ref{sec:weighted_random}, we study weighted random complexes} and prove our point process convergence results. In subsection~\ref{sec:uniformly_Weighted_Complex}, the weights are independent and identically distributed (i.i.d.) uniform $[0,1]$ random variables while, in subsection~\ref{sec:extensions}, the weights are either i.i.d. with a more general distribution $\sF$ or a perturbation of the same. \dy{In the Appendix, we give proofs of two results needed for the main part of the paper and a brief explanation on the method of factorial moments.}


\subsection{Simplicial complexes and minimal spanning acycles:}
\label{ssec:msa}

We begin by defining a (simplicial) complex, which is a higher dimensional analogue of a graph.
\begin{definition}
\label{def:complex}
An \emph{(abstract) simplical complex} $\K$ on a finite {\em ground set} $V$ is a collection of subsets of $V$ such that if $\sigma_1 \in \K$ and $\emptyset \neq \sigma_2 \subset \sigma_1$, then $\sigma_2 \in \K$ as well. The elements of $\K$ are called {\em simplices or faces} and the dimension of a simplex $\sigma$ is $|\sigma| - 1,$ where $|\cdot|$ means cardinality. A $d-$face of $\K$ is a face of $\K$ with dimension $d.$
\end{definition}
Given a complex $\K$ and $d \geq 0,$ we denote the $d-$faces of $\K$ by $\F^d(\K)$ and its $d-$skeleton by $\K^d$ (i.e., the sub-complex of $\K$ consisting of all faces of dimension at most $d$). We use $\sigma, \tau$ to denote faces and the dimension of the face shall not be explicitly mentioned unless required. A graph is a complex that consists only of $0$-faces and $1$-faces, or in other words, the $1$-skeleton of a complex is a graph. Associated to each simplicial complex is a collection of non-negative integers denoted $\beta_0(\K),\beta_1(\K),\ldots,$ called the Betti numbers\footnote{Throughout the paper, we work with reduced Betti numbers defined using field coefficients.} (see Section \ref{sec:topology} for detailed definitions) which are a measure of connectivity of the simplicial complex. Informally, the $d-$th Betti number counts the number of $(d + 1)$-dimensional holes in the complex or equivalently the number of independent non-trivial cycles formed by $d-$faces. Two points to note at the moment are: (i) $\beta_0(\K)$ is one less than the number of connected components in the graph formed by  $0$-faces and $1$-faces and (ii) if the dimension of $\K$ (maximum of dimension of faces) is $d$, then $\beta_j(\K) = 0$ for all $j \geq d+1$. \\

The Betti numbers described above are closely connected to spanning acycles. For example, the spanning tree of a graph on a vertex set $V$ can be described in topological terms as a set of edges $S$ such that $\beta_0(V \cup S) = \beta_1(V \cup S) = 0,$ i.e., $V \cup S$ is connected and has no cycles. The following higher-dimensional generalization by Kalai \cite{kalai1983enumeration} is then natural.

\begin{definition}[Spanning and Maximal acycle]
\label{def:SA}
Consider a complex $\K$ of dimension at least $d$, $d \geq 1$. \gt{A subset $S$ of $d-$faces is said to be {\em spanning} if $\beta_{d-1}(\K^{d-1} \cup S) = 0$ and an {\em acycle} if $\beta_{d}(\K^{d-1} \cup S) = 0;$ it is called a {\em spanning acycle} if it has both the properties. Separately,} a subset $S$ of $d-$faces is called a {\em maximal acycle} if it is an acycle and maximal with respect to (w.r.t.) the inclusion of $d-$faces.
\end{definition}

Though this definition of a spanning acycle merely replaces appropriate indices in the definition of a spanning tree, what is not obvious is that this is a \emph{good} higher-dimensional generalization of a spanning tree. \gt{This work of ours is the first to formally ascertain that several key properties of a spanning tree naturally extend to a spanning acycle as well; see our results in Section~\ref{sec:MSA}.}

\gt{An alternative but} more explicit algebraic description of a spanning tree is that  it consists of a set of columns which form a basis for the column space of the incidence matrix or \gt{the} boundary matrix; i.e., the matrix  $\partial_1$ whose rows are indexed by vertices and columns by edges \remove{such that the} \gt{and its} $i,j-$th entry is $1$ if the vertex $i$ belongs to the edge $j$ and $0$ otherwise. For simplicity, we \remove{are assuming our} \gt{assume throughout this paper that the} underlying field $\bF = \bZ_2$ here, i.e., all vector spaces involved are $\bZ_2$-vector spaces. It is well-known that the space of bases for these vector spaces form a matroid. Such a description also holds for spanning acycles. \gt{While we never explicitly work with this latter description in this paper, it, however, implicitly underpins many of our proof ideas. We shall explicitly point this out whenever that is the case.} 

\remove{
\red{It is helpful to consider the definition of a spanning tree via Betti numbers (see Section~\ref{sec:topology} for a more complete introduction). The spanning tree is a set of columns which form a basis for the column space of the incidence matrix or boundary matrix; i.e., the matrix $\partial_1$ whose rows are indexed by vertices and columns by edges such that the $i,j-$th entry is $1$ if the vertex $i$ belongs to the edge $j$ and $0$ otherwise\footnote{For simplicity, we are assuming our underlying field $\bF = \bZ_2$ here, i.e., all vector spaces involved are $\bZ_2$-vector spaces.}. 
}}

If we assign weights to the faces, we obtain a weighted complex $\K.$ Now, setting $w(S) := \sum_{\sigma \in S} w(\sigma)$ for a subset $S$ of simplices, we can naturally define a minimal spanning acycle as a spanning acycle $S$ with minimum weight $w(S)$. Since we deal with only finite complexes, the existence of a minimal spanning acycle is guaranteed once a spanning acycle exists. We shall denote a minimal spanning acycle by $M_d$ or simply $M$ when the dimension is clear. Though Kalai's definition of a spanning acycle and enumeration of number of spanning acycles (a generalization of Cayley's formula for spanning trees) is more than three decades old, it is receiving increased attention in the last few years
\cite{Bajo14,Catanzaro15,Duval09,Duval11,Duval16,Kalisnik17,Krushkal14,hiraoka2015minimum,Hiraoka16,Lyons09,Linial14shadows,Mathew15Hypertrees}.
%
In  Section~\ref{sec:MSA}, we  prove some fundamental properties for minimal spanning acycles: existence, uniqueness, cut property and a simplicial Kruskal's algorithm.  
Here we would also like to emphasize that properties of spanning acycles are preserved under simplicial isomorphisms but not necessarily under homotopy equivalence. 

\dy{We would like to \gt{highlight} that some more fundamental properties of the minimal spanning acycles can be found in an earlier version of our article (see~\cite{Skraba17}); we do not include them here since they are not used elsewhere in the paper. Of these, we would like to point the reader to two interesting results which are not known for matroids in general. First is an inclusion-exclusion identity for the cardinality of maximal acycles which is derived using the Mayer-Vietoris exact sequence from algebraic topology. Second, we provide a generalization of Jarn\'ik-Prim-Dijkstra's algorithm to spanning acycles. In fact, we need the complex to be `hypergraph connected' for the Prim's algorithm to work and it is not obvious what is the analogous notion of `hypergraph connectivity' in general matroids. As part of the proof, we also show that a spanning acycle is `hypergraph connected', again by using the Mayer-Vietoris sequence.}

\subsection{Persistence diagrams and minimal spanning acycles.}
\label{ssec:pd-msa}
%
We now preview the connection between persistence diagrams and acycles.  \gt{Let $\K$ be a weighted complex such that the real-valued weight function $w$ is monotone. Then, $\K(t) := w^{-1}(-\infty,t]$ is a simplicial complex for all $t \in \Real$ and we will refer to $\{\K(t): t \in \Real\}$ as the filtration induced by $w$ on $\K.$}

Let $d \geq 0$ and suppose that $\beta_d(\K) = 0$. Let $\beta_d(t) = \beta_d(\K(t)).$ We remark that $\beta_d(t)$ is a jump function. The times of positive jumps (counted with multiplicity) are birth times $\mathcal{B} = \{B_i\}$ of the persistence diagram and the times of negative jumps (counted with multiplicity) are death times $\mathcal{D} = \{D_i\}.$ The correct way to count multiplicity will be made clear in Definition \ref{def:pd}. However, if the weight function is injective, then there is no multiplicity. The non-expert reader may assume weight functions to be injective for ease of understanding the results in the introduction.

\gt{Formally, a persistence diagram corresponding to dimension $d$ is the multiset of the points $\{(B_i,D_i)\}.$ Note that it is not only a record of the birth and death times, but importantly also of the pairing of a birth with its corresponding death. A persistence diagram is useful for understanding} the evolution of topology of a filtration. See Figure \ref{fig:lm_example} for persistence diagrams of two weighted random complexes - the uniformly weighted random $d-$complexes (see Section~\ref{sec:uniformly_Weighted_Complex}) and Erd\H{o}s-R\'enyi clique complexes. The aforementioned persistence diagram would be referred to as the persistence diagram of $\Hg_d(\K)$ whenever we wish to avoid ambiguities about the dimension and the underlying complex. In this paper, we shall focus only on their two projections - birth and death times. Though not everything can be inferred from these projections, a crucial quantity that can be understood from these projections is the {\em lifetime sum} $L_d(\K) := \sum_i (D_i - B_i),$ \gt{which by Fubini's theorem also equals $\int_{0}^\infty \beta_d(t) \df{t}$ (\cite[(1.4)]{hiraoka2015minimum}).} We now  present the \dy{first of our main theorems that connects persistence diagrams} to minimal spanning acycles. Here and elsewhere, when the underlying complex $\K$ is clear we shall drop it from all our notations.
\begin{figure}[th!]
\centering
\includegraphics[width = \textwidth]{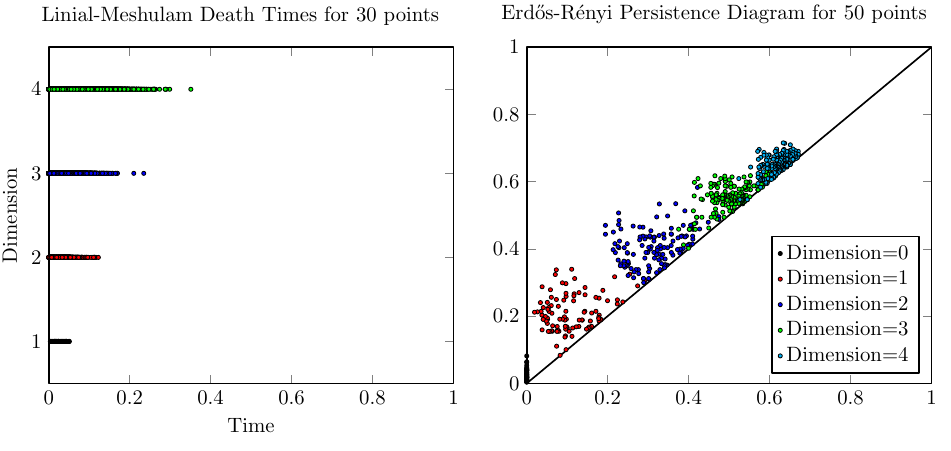}
\caption{(Left) The death times corresponding to the uniformly weighted random $d-$complex built on 30 points in different dimensions. (Right) The persistence diagram for an Erd\H{o}s-R\`enyi clique complex for 50 points. }
\label{fig:lm_example}
\end{figure}
\begin{theorem}
\label{thm:death_MSA}
Let $\K$ be a weighted $d-$complex with $\beta_{d-1}(\K) = 0$. Let $\mathcal{D}$ be the point-set of death times in the persistence diagram of the $\Hg_{d-1}(\K)$ with the canonical filtration\footnote{Again, this terminology will be explained in Section \ref{sec:PersistentHomology} and is required only for non-injective weight functions.} induced by the weights. Similarly, let $\mathcal{B}$ be the point-set of birth times in the persistence diagram of $\Hg_{d}(\K).$ Then, we have that
\[ \mathcal{D} = \{w(\sigma) : \sigma \in M \} \quad \text{ and } \quad \mathcal{B} = \{w(\sigma) : \sigma \in \F^{d} \backslash M \},\]
where $M$ is a $d-$minimal spanning acycle of $\K$ and $\F^{d}$ are the $d-$simplices of $\K$.
\end{theorem}
%
\gt{This result reveals a stronger connection between persistence diagrams and minimal spanning acycles than what is known in literature}. If $\K$ is a weighted $d-$complex with $\beta_{d-1}(\K) = \beta_{d-2}(\K) = 0$ then, as a corollary of the above theorem, we obtain the following relation
\gt{
\begin{equation}
\label{eqn:lifetime_Hiraoka}
L_{d-1} = \sum_i(D_i - B_i) =  
\sum_{\sigma \in M_d}w(\sigma)   -  \sum_{\sigma \in \F^{d-1} \setminus M_{d-1}}w(\sigma) = w(M_d) + w(M_{d-1}) - w(\F^{d-1}).
\end{equation}}
For $d = 1$ (assuming $\K^0 \subset w^{-1}(0)$), the above relation is well known and, for $d \geq 2$, this relation was derived recently in \cite[Theorem 1.1]{hiraoka2015minimum} using different techniques. This latter paper and, in particular, their derivation of \eqref{eqn:lifetime_Hiraoka} served as our stimulus to investigate minimal spanning acycles. Apart from its striking simplicity, 
we believe Theorem \ref{thm:death_MSA} can be useful in studying either of them using the other. In fact, this result is frequently used in this paper. Much of the complexity in understanding persistent homology arises from the pairing of birth and death times. The above result is useful in understanding death or birth times individually and, in certain cases, this shall yield useful information (e.g., lifetime sum) even without the knowledge of the pairings.
The proof of the above theorem and some of its consequences can be found in Section \ref{sec:MSA_persistence}. 


\subsection{ Stability of birth and death times}
\label{ssec:stability}

Stability results (e.g. \cite[Section VIII.2]{Edelsbrunner10}, \cite{chazal2009proximity,chazal2014persistence,cohen2007stability,cohen2010lipschitz}) are an important cog in the wheel of topological data analysis and provide a theoretical justification for the robustness of persistent homology. While $L_{\infty}$ stability (or {\em bottleneck stability}) is the most standard form of stability proven for persistence diagrams, $L_p$ stability for $p \geq 0$ requires restrictive assumptions that are not widely applicable. Using  simplicial version of Kruskal's algorithm and the correspondence (Theorem \ref{thm:death_MSA}), we prove the following stability result  separately for the birth and death times with minimal assumptions. The usefulness of this stability result will become apparent in Section~\ref{ssec:rc}
\begin{theorem}
\label{thm:l1_stability}
Let $\K$ be a finite complex with two weight functions \gt{$f, f';$} both of which induce a filtration on $\K.$ Let $\mathcal{B}_f  = \{B_i\}$ and $\mathcal{D}_f = \{D_i\}$ be the respective birth and death times in the $\Hg_d(\cdot)$ and $\Hg_{d-1}(\cdot)$ persistence diagrams of $f.$ Similarly, define  \gt{$\mathcal{B}_{f'}  =  \{B'_i\}$} and \gt{$\mathcal{D}_{f'} = \{D'_i\}$ w.r.t. $f'.$} Let $\Pi_D$ be the set of bijections from $\mathcal{D}_f$ to \gt{$\mathcal{D}_{f'}$} and, similarly, let $\Pi_B$ be the set of bijections from $\mathcal{B}_f$ to \gt{$\mathcal{B}_{f'}$}. Then, for any $p \in \{0,\ldots,\infty\}$,
\[   \max\{\textstyle{\inf_{\pi \in \Pi_D}}\sum_i|D_i - \pi(D_i)|^p, \textstyle{\inf_{\pi \in \Pi_B}} \sum_i|B_i - \pi(B_i)|^p\} \leq \sum_{\sigma \in \F^d}|f(\sigma) - f'(\sigma)|^p,\]
%
%
For $p = \infty$ and a sequence $\{x_i\}_{i \geq 1}$, in the usual manner, $\sum_i |x_i|^p$ should be read as $\sup_i |x_i|$.
\end{theorem}
%

As part of the proof (see Section \ref{sec:stability}) of the above stability result, we show that \emph{on a fixed simplicial complex} changing weights of $m$ ($m \geq 1$) faces can change at most $m$ death times and $m$ birth times by the difference between the weights on the faces\footnote{By fixing the underlying space, we can ensure the cardinalities of birth and death times remain the same.}. One might suspect that the $L_{\infty}$ stability in the above theorem can be deduced from the bottleneck stability of persistence diagrams by a projection argument. 
This is, however, not the case due to the fact that the diagonal plays a special role in the definition of bottleneck stability of persistence diagrams, \gt{but for point processes on $\bR$ there is no equivalent to the diagonal}.

\subsection{Weighted random complexes}
\label{ssec:rc}

Having offered a teaser to our deterministic results, we now turn to a preview of the probabilistic results. Whereas there is a rich recent literature on deterministic aspects of spanning acycles (see in Section \ref{ssec:msa}) and random complexes (see below), the literature is sparser on weighted random complexes or random minimal spanning acycles. The probabilistic model of interest to us is the one introduced by Linial and Meshulam \cite{linial2006homological} and then extended by Meshulam and Wallach \cite{meshulam2009homological}. This model, called the random $d-$complex and denoted by $Y_{n,d}(p)$, consists of all faces on $n$ vertices (i.e., ground set $V = [n] := \{1, \ldots, n\}$) with dimension at most $(d-1)$ and each $d-$face is included with probability $p$ independently. $Y_{n,1}(p)$ is the classical Erd\H{o}s-R\`enyi graph on $n$ vertices with edge-connection probability $p$. Like Erd\H{o}s-R\`enyi graph is a mean-field model of pairwise interactions, the random $d-$complex can be considered as a model of higher-order interactions. This model has spawned a rich literature in the recent years \cite{linial2006homological,meshulam2009homological,Costa12DCG,Costa12asphericity,Costa15,Linial14phase}.  Although we focus on the random $d-$complex, we alert the reader of the existence of a richer theory of random complexes and topological data analysis \cite{Carlsson14, kahle2014topology, bobrowski2014topology, Kahle16, Costa12}.

The focus of many studies on random $d-$complexes has been the two non-trivial Betti numbers of the complex: $\beta_{d-1}(\cdot)$ and $\beta_d(\cdot)$. The starting part of our study is the following fine phase transition result for $\beta_{d-1}(Y_{n,d}(p))$.
\begin{lemma}
{\cite{stepanov1969combinatorial}, \cite[Theorem 1.10]{kahle2014inside}}
\label{lem:BettiNumber}
Fix $d \geq 1.$ Consider $Y_{n, d}(p_n)$ with
\begin{equation}
\label{eqn:pn}
p_n = \frac{d \log n + c - \log(d!)}{n}
\end{equation}
for some fixed $c \in \bR.$ Then, as $n \to \infty,$ $\beta_{d - 1}(Y_{n, d}(p_n)) \Rightarrow \Poi(e^{-c}),$ where $\Poi(\lambda)$ stands for the Poisson random variable with mean $\lambda$ and $\Rightarrow$ denotes convergence in distribution.
\end{lemma}
The proof of this result proceeds as follows: First, it is shown that $N_{n, d - 1}(p_n) \Rightarrow \Poi(e^{-c}),$ where $N_{n , d  - 1}(p)$ denotes the number of isolated $(d - 1)-$faces in $Y_{n, d}(p).$ Then, for $p_n$ as chosen, it is established that $N_{n, d - 1}(p_n)$ completely determines the behaviour of the $(d - 1)-$th Betti number \dy{(see also Appendix \ref{app:cohomology})}. Building upon this relation, one also has that $\Pr\{\beta_{d - 1}(Y_{n, d}(p_n)) = 0\} \to 1$ if $np_n - d\log n \to \infty$ and $\Pr\{\beta_{d - 1}(Y_{n, d}(p_n)) = 0\} \to 0$ if $np_n - d\log n \to -\infty.$ These were proven  by Erd\H{o}s and R\'{e}nyi \cite{Erdos59} in 1959 for $d=1$, much later by Linial and Meshulam  \cite{linial2006homological} in 2006 for $d=2$ and shortly thereafter in 2009 for $d \geq 3$ by Meshulam and Wallach \cite{meshulam2009homological}.

One of the goals of this paper is to generalize Lemma~\ref{lem:BettiNumber} first to the level of persistence diagrams and then to that of minimal spanning acycles of randomly weighted $d-$complexes. Before providing the actual statements, we give a formal definition of these weighted complexes.
\begin{definition}
\label{defn:Generically_Perturbed_Weighted}
Let $d \geq 1$ be some integer. Consider $n$ vertices and let $\K_{n}^{d}$ be the complete $d-$skeleton on them. Let $\phip : \K_{n}^{d} \to [0,1]$ be the weight function with the following properties:
\begin{enumerate}
\item $\phip(\sigma) = 0$ for $\sigma \in \bigcup_{i = 0}^{d - 1} \F^i,$ and
\item $\phip(\sigma) = \phi(\sigma) + \epsilon_n(\sigma)$ for $\sigma \in \F^{d}.$  
\end{enumerate}
Here, $\{\phi(\sigma) : \sigma \in \F^{d}\}$ are real valued i.i.d. random variables with \gt{(cumulative) distribution function $\sF : \Real \to [0, 1]$} perturbed respectively by $\{\epsilon_n(\sigma) :  \sigma \in \F^{d}\}$. The latter are  another set of real valued random variables not necessarily identically distributed or independent of each other or $\phi(\sigma)$'s.
The {\em randomly weighted $d-$complex} $\Lp_{n, d}$ is the simplicial complex $\K_n^d$ weighted by $\phip.$ Associated with $\Lp_{n, d}$ is the canonical simplicial process given by the filtration $\{\Lp_{n, d}(t): t \in \Real\},$ where $\Lp_{n, d}(t) = \{\sigma: \phip (\sigma) \leq t\}.$
\end{definition}
For ease of use, we shall write $\sigma \in \Lp_{n, d}$ to mean $\sigma \in \K_{n}^{d}.$  Similarly, $\F^i(\Lp_{n, d})$ shall mean $\F^i(\K_{n}^{d})$ and so on. Finally, let $\|\epsilon_n\|_{\infty} := \max_{\sigma \in \F^{d}(\Lp_{n, d})} |\epsilon_n(\sigma)|.$

Our key result concerning randomly weighted $d-$complexes is that if the perturbations decay sufficiently fast, then suitably scaled point processes related to the nearest face distances, weights of the faces in the $d-$minimal spanning acycle, and death times in the associated persistence diagram all converge to \dy{the same} inhomogeneous Poisson point process. The proof crucially relies upon Theorems \ref{thm:death_MSA} and \ref{thm:l1_stability}.

Formally, we consider the following {\em three scaled point processes} on $\Real.$
\begin{enumerate}
\item (Extremal) nearest face distances, i.e., $\PoiFt := \{n \sF(C'(\sigma)) - d\log n + \log(d!) : \sigma \in \F^{d - 1}\},$ where, for $\sigma \in \F^{d - 1},$
\begin{equation}
\label{eqn:Connection Time}
C'(\sigma) := \min\limits_{\tau \in \F^{d}, \tau \supset \sigma } \phip(\tau).
\end{equation}

\item (Extremal) death times in $\Hg_{d - 1},$ i.e., $\PoiDt := \{n \sF(D'_i) - d\log n + \log(d!)\},$ where $\{D'_i\}$ is the set of death times in the persistence diagram of $\Hg_{d - 1}$ (see Definition \ref{defn:death_times}).

\item (Extremal) face weights in $M',$ i.e., $\PoiMt := \{n \sF(\phip(\sigma)) - d\log n + \log(d!): \sigma \in M'\},$ where $M'$ is a $d-$minimal spanning acycle in $\Lp_{n, d}$ (see \eqref{defn:MSA}).
\end{enumerate}

\gt{Observe that  the scaling used in the definitions of each of $\PoiFt, \PoiDt,$ and $\PoiMt$ pushes quantities less than the $d \log(n)/n$ threshold to $-\infty,$ asymptotically. In that sense, asymptotically, the three processes only consider the extremal values, i.e., those that are above this threshold.}\dy{The reason for transforming weights, \gt{as will be seen below,} is that it yields a limiting point process independent of $\sF$. If we think of the weighted complex as a dynamic complex with simplices being added at times equal to their weights, then the transformation by $\sF$ is nothing but a time-change.} 

\gt{At first glance, these are three distinct point processes on $\Real$ and there are no obvious reasons why they ought to be connected. However, by applying Theorem \ref{thm:death_MSA}, we get $\PoiM = \PoiD$ and, from Corollary~\ref{cor:nng_MSA} that we establish later, it follows that $\PoiFt \subset \PoiMt.$ A natural guess based on this would be that a similar relation holds  amongst the three processes asymptotically as well. Surprisingly, the below result shows that the three processes in fact have the same asymptotic behaviour. }

\begin{theorem}
\label{thm:perturbed_process_convergence}
Suppose that $\sF$ is \gt{Lipschitz continuous}. If $n \|\epsilon_n\|_{\infty} \to 0$ in probability, then each of $\PoiFt, \PoiDt,$ and $\PoiMt$ converges vaguely in distribution to $\PoiP$, where $\PoiP$  is the Poisson point process with intensity $e^{-x}\md x$ on $\bR$.
\end{theorem}
Since the $(d - 1)-$faces have zero weights, the birth times in the persistence diagram of $\Hg_{d - 1}$ are all zero. Hence, if $\|\epsilon_n\|_{\infty} = 0$ and $\sF$ is the distribution function of $U[0, 1],$ then $\PoiDt((c,\infty)) = \beta_{d-1}(Y_{n,d}(p_n))$ for $p_n$ and $c$ as in Lemma \ref{lem:BettiNumber}. Thus, a point process convergence for $\PoiDt$ in this special case implies Lemma \ref{lem:BettiNumber} as a corollary. This and more follows from the above result. See Figure \ref{fig:lm_example}(a) for simulations of $\PoiDt$ for $d = 1,2,3,4$ in the above special case.

To the best of our knowledge, a point process convergence result as above is not known even for complete graphs with i.i.d. uniform $[0,1]$-weights, which might be considered as a mean-field model for random metric spaces. For random geometric graphs, such a point process convergence result for extremal edge weights of the minimal spanning tree was proven in \cite{Penrose97,Hsing05}. These results were important to understand the connectivity of random geometric graphs. However, reversing the scenario, we have gone from results on connectivity (i.e., $\Hg_k(\cdot)$ persistence diagrams) to those for minimal spanning acycles.

The above weak convergence result along with the continuous mapping theorem yields asymptotics of various statistics of $\PoiDt$.  Our result could be useful in deriving asymptotics for extremes of other summary statistics of persistence diagrams such as persistence landscapes \cite{Bubenik15}, homological scaffolds \cite{Petri14} or accumulative persistence function \cite{accumulative2016}.

As for our proof, we first deal with the case when $\sF$ is the distribution function of $U[0,1]$ and $\epsilon_n(\sigma) \equiv 0.$ We use the factorial moment method to show convergence of the first point process and then use cohomology theory to show that this is a good enough approximation for the second point process. This yields convergence of the second point process . Finally, this along with Theorem~\ref{thm:death_MSA} gives the convergence of the third point process (see Section \ref{sec:uniformly_Weighted_Complex}). This approach is inspired by those of \cite{linial2006homological, meshulam2009homological, kahle2014inside}. Next, we extend this result to the case of the more general i.i.d. weights. Finally, we complete the proof of Theorem \ref{thm:perturbed_process_convergence} by using our stability result (Theorem \ref{thm:l1_stability}) as well as showing that the topology of bottleneck distance between Radon counting measures is stronger than vague topology (see Section \ref{sec:extensions}).
%

We now present one more powerful consequence of our stability result. While it is believed that introducing weak dependencies between the random variables should not affect the asymptotics, it is  often difficult to prove such a statement rigorously. As we \gt{again illustrate}, our stability result helps bridge this gap in certain situations. In particular, given an arbitrary random complex, it enables one to translate certain limit theorems to noisy variants of this complex once the same has been shown in the noiseless setting.

Consider $\Lp_{n, d}$ from Definition~\ref{defn:Generically_Perturbed_Weighted} and suppose that $\sF$ is the distribution function of $U[0, 1].$ Further, let $\Dgm(\K_n^d,\phip) = \{(0,D'_i)\}$ be the $\Hg_{d-1}$ persistence diagram. Let us define the (weighted) lifetime sums for $\alpha \geq 0$ as
\begin{equation}
\label{eqn:lifetimeSum}
(L_{n,d-1}')^{\alpha} = \sum_i (D'_i)^{\alpha}.
\end{equation}
To begin with, suppose that $\|\epsilon_n\|_{\infty} = 0$ for all $n \geq 1$. \dy{In such a case, we denote the weighted random complex by $\U_{n,d}$ and the corresponding lifetime sum by $L_{n,d}$.} Then, it follows from a remarkable recent result by Hino-Kanazawa (\cite[Theorem 4.11]{Hino2018}) that, for $\alpha > 0$
\begin{equation}
\label{eqn:hino}
n^{-(d-\alpha)}\Exp[(L_{n,d-1})^{\alpha}] \to I^{\alpha}_{d-1},
\end{equation}
where $I^{\alpha}_{d-1}$ is an explicitly defined constant (see \cite[(4.10)]{Hino2018} for the definition of constants and \cite[Section 4.4]{Hino2018} for more concrete expressions).  In the special case of $d = 1, \alpha = 1$, this is the famed result of Frieze \cite{Frieze85} for random minimal spanning trees with $I^1_0 = \zeta(3)$ where $\zeta$ is the Riemann-zeta function. Further, $I^p_0$ for $p \in \{1,2,\ldots\}$ are shown to be linear combinations of $\zeta(3), \zeta(4),$ etc. Using our stability result, we now extend this result to the noisy case. The proof can be found in Section \ref{sec:extensions}.
\begin{corollary}
\label{thm:lifetimesum}
Fix a $p,d \in \{1,2,\ldots\}$. Assume that $\sF$ is the distribution function of $U[0,1]$ and that $ \sup_{\sigma \in \F^d} \Exp[|\epsilon_n(\sigma)|^p] = o(n^{-(p+1)})$. Then, with $I^p_{d-1}$ as defined in \eqref{eqn:hino}, we have that
\[ n^{-(d-p)}\Exp[(L')^p_{n,d-1}] \to I^p_{d-1}.\]
\end{corollary}

\section{Preliminaries}
\label{sec:prelim}

We describe here the basic notions of simplicial homology, persistent homology, and point processes. We remark that, in an earlier version of the paper (see \cite[Appendix B]{Skraba17}), we have rephrased our topological notions in the language of matrices for \dy{an alternative and computationally convenient viewpoint.}

\subsection{Topological notions}
\label{sec:topology}

We point out that we shall always choose our coefficients from a field $\bF$. In this regard, $0$ stands for additive identity, $1$ stands for multiplicative identity and $-1$ for the additive inverse of $1$. An often convenient choice in computational topology is $\bF = \bZ_2$ in which case $1 = -1$.

\subsubsection{Simplicial Homology}
\label{sec:simplicial_homology}
For a good introduction to algebraic topology, see \cite{Hatcher02}, and for simplicial complexes and homology, see \cite{Edelsbrunner10,Munkres84}.

Let $\K$ be a simplicial complex (see Definition~\ref{def:complex}). We assume throughout that all our simplicial complexes are defined over a finite set $V$. 
The $0$-faces of $\K$ are also called as {\em vertices}. When obvious, we shall omit the reference to the underlying complex $\K$ in the notation. A $d-$simplex $\sigma$ is often represented as $[v_0,\ldots,v_d]$ to explicitly indicate the subset of $V$ generating the simplex $\sigma$.

An {\em orientation} of a $d-$simplex is given by an ordering of the vertices and denoted by $[v_0, \ldots, v_d].$ Two orderings induce the same orientation if and only if they differ by an even permutation of the vertices. In other words, for a permutation $\pi$ on $[d]$,
$$ [v_0,\ldots,v_d] = (-1)^{sgn(\pi)}[v_{\pi(0)},\ldots,v_{\pi(d)}], $$
\gt{where $sgn(\pi)$ denotes the sign of the permutation $\pi.$}
We assume that each simplex in our complex is assigned a specific orientation (i.e., ordering).

Let $\bF$ be a field. A simplicial $d-$chain is a formal sum of oriented $d-$simplices $\sum\limits_{i} c_i \sigma_i, c_i \in \bF.$ The free abelian group generated by the $d-$chains \gt{is called the $d-$th chain group and is denoted by  $C_d(\K).$ Formally,} 
\[
C_d(\K) := \left\{ \sum_ic_i \sigma_i : c_i \in \bF, \sigma_i \in \F^d(\K) \right\}.
\]
Clearly, $C_d(\K)$ is a $\bF$-vector space. We shall set $C_{-1} = \bF$ and $C_d = 0$ for $d = -2,-3,\ldots$. For a vector space, let $\beta(\cdot)$ denote its rank. Thus, $\beta(C_d) = f_d$ for $d \geq 0$. For $d \geq 1, $ we define the boundary operator $\bdr_d : C_d \To C_{d - 1}$ first on each $d-$simplex using
$$ \bdr_d([v_0, \ldots, v_d]) = \sum_{i=0}^d(-1)^i[v_0,\ldots,\hat{v_i},\ldots,v_d],$$
and then extend it linearly on $C_d.$ Above, $\hat{v_i}$ denotes that $v_i$ is to be omitted. $\bdr_0$ is defined by setting $\bdr_0([v]) = 1$ for all $v \in \F^0$. It can be verified that $\bdr_d$ is a linear map of vector spaces and more importantly that $\bdr_{d-1} \circ \bdr_d = 0$ for all $d \geq 1,$ i.e., boundary of a boundary is zero. When the context is clear, we will drop the dimension $d$ from the subscript of $\partial_d.$

Note that the free abelian group of $d-$chains is defined only using $\F^d(\K)$. When we use a subset $S \subset \F^d(\K)$ of $d-$faces rather than the entire collection of $d-$faces to generate the free abelian group, we shall use $C_d(S)$ to denote the corresponding free abelian (sub)group of $d-$chains. In other words, $C_d(S) = C_d(\K^{d-1} \cup S)$. \remove{Also, for a $d-$chain $\sum_i a_i \sigma_i$, the support denoted as $supp(\sum_i a_i \sigma_i)$ is $\{\sigma_i : a_i \neq 0\} \subset \F^d$. For  $\mathbb{F} = \mathbb{Z}_2$, the support characterizes the $d-$chain.}

The $d-$th {\em boundary space} denoted by $B_d$ is $\im \bdr_{d + 1}$ and the $d-$th {\em cycle space} $Z_d$ is $\ker \bdr_d$.  Elements of $Z_d$ are called {\em cycles} or {\em $d-$cycles} to be more specific. The $d-$dimensional (reduced)\footnote{Reduced is used to refer to the convention that $C_{-1} = \bF$ instead of $C_{-1} = 0$.} homology group is then defined as the quotient group
\be
\label{eqn:hom_gp}
\Hg_d = \frac{Z_d}{B_{d}}.
\ee
Again, since we are working with field coefficients, $B_d,Z_d$ and $\Hg_d$ are all $\bF-$vector spaces. \red{The bases of these vector spaces form a matroid (\cite{Oxley03,Welsh76}). This implies that certain concepts such as the span of a generating set and properties such as the exchange property automatically hold. While it is not necessary for understanding our results, a familiarity with matroids is helpful.}

The $d-$th {\em Betti number} of the complex $\beta_d(\K)$ is defined to be the rank of the vector space $\Hg_d$. Respectively, let $b_d(\K) := \beta(B_d)$ and $z_d(\K) :=  \beta(Z_d)$ denote the ranks of the $d-$th boundary and $d-$th cycle spaces, respectively. Thus, we have that $\beta_d = z_d - b_d.$ Note that we drop the adjective reduced henceforth, but all our homology groups and Betti numbers are indeed reduced ones. Some authors prefer to use $\tilde{H}_d$ and $\tilde{\beta}_d$ to denote reduced homology groups and Betti numbers respectively, but we refrain from doing so for notational convenience. However, under such a notation, we note that $\beta_d - \tilde{\beta}_d = 1[d = 0]$. This gives an easy way to translate results for reduced Betti numbers to Betti numbers and vice-versa. We denote the Euler-Poincar\'{e} characteristic by $\chi$ and the Euler-Poincar\'{e} formula holds as follows:
\begin{equation}
\label{eqn:EP_formula}
\chi(\K) = \sum_{j=0}^{\infty}(-1)^j f_j(\K) = 1 + \sum_{j=0}^{\infty}(-1)^j \beta_j(\K).
\end{equation}

An important property of homology groups that is often of use is the following: If $\K_1,\K_2$ are two complexes such that the function $h : \K^0_1 \to \K^0_2$ is a {\em simplicial map} (i.e., $\sigma_1 =[v_0,\ldots,v_d] \in \K_1$ implies that $h(\sigma_1) = [h(v_0),\ldots,h(v_d)] \in \K_2$ for all $d \geq 0$), then there exists an homomorphism $h_* : \Hg_d(\K_1) \to \Hg_d(\K_2)$ called the {\em induced homomorphism} between the homology groups. \gt{If $\K_1 \subseteq \K_2,$ then a natural simplicial map is the inclusion map from $\K_1$ to $\K_2.$ The case} \red{of multiple inclusions} \gt{now} \red{brings us to persistent homology.}
\remove{A tool from algebraic topology that we use here is the {\em Mayer-Vietoris sequence}. Below we state it explicitly and its implication for Betti numbers.

A sequence of vector spaces $V_1,\ldots,V_l$ and linear maps $\nu_i: V_i \to V_{i+1}$,  $i=1,\ldots,l-1$ is said to be {\em exact} if $\im \,\nu_i = \ker \, \nu_{i+1}$ for all $i = 1,\ldots,l-1$.
\begin{lemma}(Mayer-Vietoris Sequence \cite[Theorem 25.1]{Munkres84} , \cite[Corollary 2.2]{delfinado1993incremental}.) \label{lem:MV_complexes} \\
\noindent Let $\K_1$ and $\K_2$ be two finite simplicial complexes and $\K_0 = \K_1 \cap \K_2$ (i.e., $\K_0$ is the complex formed from all the simplices that are in both $\K_1$ and $\K_2$). Then the following are true:
\begin{enumerate}
\item The following is an exact sequence, and, furthermore, the homomorphisms $\nu_d$ are  induced by the respective inclusions:
\begin{align*}
\label{eqn:MV}
& \cdots\to \Hg_d(\K_0) \stackrel{\nu_d}{\to} \Hg_d(\K_1)\oplus \Hg_d(\K_2) \to \Hg_d(\K_1 \cup \K_2)
\\ &\qquad\qquad\qquad \qquad\qquad\qquad
\to \Hg_{d-1}(\K_0) \stackrel{\nu_{d-1}}{\to} \Hg_{d-1}(\K_1)\oplus \Hg_{d-1}(\K_2)\to \cdots
\end{align*}

\item Furthermore,
\be
\label{eqn:MV_identity}
\beta_d(\K_1 \cup \K_2) \ = \ \beta_d(\K_1) + \beta_d(\K_2)+ \beta(\ker\nu_d) +\beta(\ker \nu_{d-1})-\beta_d(\K_0).
\ee

\end{enumerate}
\end{lemma}
}
\subsubsection{Persistent Homology}
\label{sec:PersistentHomology}

A \emph{filtration} of a simplicial complex $\K$ is a sequence of subcomplexes $\{\K(t) : t \in \overline{\Real}\}$ satisfying
\begin{equation*}
\emptyset = \K(-\infty) \subseteq \K(t_1) \subseteq \K(t_2) \subseteq \K(\infty) = \K \qquad
\end{equation*}
for all $ -\infty < t_1  \leq t_2 < \infty.$ Put differently, the filtration $\{\K(t): t \in \overline{\Real}\}$ describes how to build $\K$ by adding collections of  simplices at a time. For more complete introduction and survey of persistent homology, see \cite{Edelsbrunner10,Carlsson09,Carlsson14}.

We now describe the natural filtration associated with \emph{weighted simplicial complexes}. Consider a simplicial complex $\K$ weighted by $w: \K \to \Real$ satisfying $w(\sigma) \leq w(\tau),$ whenever $\sigma, \tau \in \K$ and $\sigma \subset \tau.$ Functions having this property are called {\em monotonic functions} in \cite[Chapter VIII]{Edelsbrunner10}. As $w$ is monotone, $\{\K(t) : t \in \overline{\Real}\}$ with $ \K(t) := w^{-1}(-\infty, t]$ forms a \emph{sublevel set} filtration of $\K.$ Further note that $w$ induces a partial order on the faces of $\K.$ Assuming axiom of choice, this partial order can always be extended to a total order~\cite{szpilrajn1930extension}. Let $<_l$ denote one such total order. We make the standing assumption that for a given weight function $w$, the same total order $<_l$ is chosen and used throughout the paper.

One can now view the above sublevel set filtration associated with $(\K, w)$ in a dynamic fashion: as the parameter $t$ evolves over $\overline{\Real},$ $\K$ gets built one face at a time respecting the total order $<_l.$ In this way, with the addition of faces, the topology of $\K$ evolves. Clearly
\begin{equation}\label{eq:definition_minus}
\K(\sigma^-) : = \{\sigma_1 \in \K: w(\sigma_1) <_l w(\sigma)\}
\end{equation}
denotes the complex right before the face $\sigma$ is to be added. Thus, given a monotonic weight function $w$, we can construct a filtration with respect to the chosen total order $<_l$. We shall call this filtration the {\em canonical filtration} associated with the total order $<_l$ or a {\em linear filtration} 	of the weight function $w$.

To track the changes in topology, akin to the definition for homology given in \eqref{eqn:hom_gp}, we define the $(t_1, t_2)$-\emph{persistent homology group} as the quotient group
$$
\Hg_d^{t_1,t_2} = \frac{Z_d^{t_1}}{Z_d^{t_1}\cap B_d^{t_2}}, \qquad t_1 \leq t_2.
$$

The information for all pairs $(t_1, t_2)$ can be encoded in a unique interval representation called a $\emph{persistence barcode}$ \cite{zomorodian2005computing} or equivalently a $\emph{persistence diagram}$ \cite{cohen2007stability}. Before giving the definition, we first note that for a finite simplicial complex endowed with a total ordering $<_l$, we can reindex the filtration by assigning a natural number to each simplex. We refer to this as a \emph{discrete weight} $w_\discrete$ corresponding to the monotonic weight function $w$ i.e., $w_\discrete (\sigma) < w_\discrete (\tau)$ iff $w(\sigma) <_l  w(\tau)$. Note that there is a bijection between total orders $<_l$ and weight function $w_\discrete$. Thus, the discrete weight has a natural, well-defined projection $\pi$ back to the original function values,
$$ \pi(i)  \mapsto w(\sigma) | w_\discrete(\sigma) = i $$
\begin{definition}
\label{def:pd}
Given a simplicial complex $\K$ with a monotonic function $w$ and the corresponding discrete weight $w_\discrete: \K \rightarrow \mathbb{N}$,  the $d-$th persistence diagram $\Dgm(\K, w_\discrete)$ is the multiset of points in the extended grid $\overline{\mathbb{N}}^2 $ such that the each point $(i, j)$ in the diagram represents a distinct class (i.e., a topological feature) in $ \Hg_d(w_\discrete^{-1}(-\infty,t])$ for all $t \in[i, j)$ and is not a class in $t \notin [i, j)$. The persistence diagram $\Dgm(\K, w)$ is then defined as the projection of the multiset of points under $\pi$, i.e., $(i,j)\in \Dgm(\K, w_\discrete)$ iff  $(\pi(i),\pi(j)) \in \Dgm(\K,w)$.
\end{definition}

This differs from the typical definition of a persistence diagram, where the existence and uniqueness of the persistence diagram  is defined in terms of an algebraic decomposition into interval modules see ~\cite{chazal2012structure,crawley2015decomposition}. For technical reasons, this approach generally discards the points on the diagonal, i.e., topological features which are both born and die at time $t$.  In the above definition, the total order guarantees that there are no points on the diagonal of the discrete filtration. However, since we deal with the restricted setting of piece-wise constant functions on finite simplicial complexes, we do not lose any information; indeed, we keep more of the chain level information. We then transform the persistence diagram back to the original monotone function.  After  the transformation, points may lie on the diagonal and, as we shall see, we do require these points.

Our definition is used implicitly in ~\cite{zomorodian2005computing}, which first identified the algebraic decomposition as a consequence of the structure theorem of finitely generated modules over a principle ideal domain. This applies in this setting since the homology groups of finite simplicial complexes are always finitely generated.  Therefore, we could have equivalently defined the diagram using the decomposition directly as done in Corollary 3.1 in ~\cite{zomorodian2005computing}, as the modified Smith Normal Form of the boundary operator~\cite{skraba2013persistence}. We believe that our definition is more accessible to a non-algebraic audience and is included for completeness. But more important for us are birth and death times defined below.
\begin{definition}
\label{defn:death_times}
The \emph{death times} (respectively \emph{birth times}) of the filtration associated with $(\K, w)$ are equal to the multiset of $y$-coordinates ($x$-coordinates) of points in $\Dgm(\K, w)$.
\end{definition}
We now discuss the notion of negative and positive faces which are vital to our proofs. 
\begin{lemma}(\cite[Section 3]{delfinado1993incremental})
\label{lem:delfinado}
Let $\K$ be a simplicial complex on vertex set $V$ and $\sigma \subset V$ be a set of cardinality $d + 1$ in $V$ for some $d \geq 0.$ Additionally, assume that $\sigma \notin \K$ but $\partial \sigma \in C_{d-1}(\K)$. Then, $\beta_j(\K \cup \sigma) = \beta_j(\K)$ for all $j \notin \{d - 1, d\}.$ Further, one and only one of the following two statements hold:
\begin{enumerate}
\item $\beta_{d-1}(\K \cup \sigma) = \beta_{d-1}(\K) -1$ and $\beta_{d}(\K \cup \sigma) = \beta_d(\K).$
\item $\beta_d(\K \cup \sigma) = \beta_d(\K) + 1$ and $\beta_{d - 1}(\K \cup \sigma) = \beta_{d - 1}(\K).$
\end{enumerate}
\end{lemma}

From the definition of the cycle and boundary spaces, the above two numbered statements can be interpreted equivalently in the following manner which shall be useful for us:
\begin{eqnarray}
\beta_{d-1}(\K \cup \sigma) = \beta_{d-1}(\K) -1 & \Leftrightarrow & b_{d - 1}(\K \cup \sigma) = b_{d - 1}(\K) + 1 \Leftrightarrow \partial \sigma \notin \partial (C_d(\K)), \label{eqn:negative_image} \\
\beta_d(\K \cup \sigma) = \beta_d(\K) + 1 & \Leftrightarrow & z_d(\K \cup \sigma) = z_d(\K) + 1 \Leftrightarrow \partial \sigma \in \partial (C_d(\K)).
\label{eqn:positive_image}
\end{eqnarray}

\remove{
\begin{remark}
\label{rem:dAndHighFacesUnaffectBettid-1}
Let $\K$ be a simplicial complex with vertex set $V.$ Let $\sigma \subset V$ be such that $|\sigma| = d + 1,$ $\sigma \notin \K$ but $\partial \sigma \in C_{d-1}(\K).$ Then a trivial consequence of Lemma \ref{lem:delfinado} is that $\beta_{d - 1}(\K) \geq \beta_{d - 1}(\K \cup \sigma),$ i.e., addition of a $d-$face cannot increase $\beta_{d-1}.$ Now let $\sigma \subset V$ be such that $|\sigma| = d^\prime + 1,$ where $d^\prime > d,$ $\sigma \notin \K$ but $\partial \sigma \in C_{d^\prime - 1}(\K).$ Then, from Lemma~\ref{lem:delfinado}, it also follows that $\beta_{d - 1}(\K) = \beta_{d - 1}(\K \cup \sigma ) ,$ i.e, addition of a face with dimension higher than $d$ cannot affect $\beta_{d - 1}.$
Given $(\K, w),$ where $w$ induces a filtration, we emphasize that the labelling of faces as either positive or negative is unique.
\end{remark}
}

\begin{definition}[Positive and Negative faces]
\label{def:positive}
Let $\K$ be a complex with vertex set $V$ and $\sigma \subset V$ be a set of cardinality $d+1$ for some $d \geq
0$. Further assume that $\sigma \notin \K$ but $\partial \sigma \in C_{d-1}(\K) \neq 0$. Such a $\sigma$ is called a {\em negative face} w.r.t. $\K$ if
$\beta_{d-1}(\K \cup \sigma) = \beta_{d-1}(\K)- 1$, it is called a {\em positive face}
if it is not negative, i.e., 	$\beta_d(\K \cup \sigma) = \beta_d(\K) + 1$.
\end{definition}

This is useful for understanding how the topology evolves in the filtration associated with $w$ (recall \eqref{eq:definition_minus}).  If $\sigma$ is a $d-$face, then Lemma~\ref{lem:delfinado} shows that the relationship between the topology of the setup before and after addition of $\sigma$ is as follows: (i) $\beta_j(\K(\sigma^-) \cup \sigma) = \beta_j(\K(\sigma^-))$ for all $j \notin \{d, d - 1\},$ (ii) one and exactly one of the following is true:
\begin{equation}
\label{Defn:NegFace}
\beta_{d - 1}(\K(\sigma^-) \cup \sigma) = \beta_{d - 1}(\K(\sigma^-))-1
\end{equation}
or
\begin{equation}
\label{Defn:PosFace}
\beta_d(\K(\sigma^-) \cup \sigma) = \beta_d(\K(\sigma^-)) + 1.
\end{equation}
As in Definition~\ref{def:positive}, when \eqref{Defn:NegFace} holds (respectively \eqref{Defn:PosFace} holds) $\sigma$ will be called a negative face (positive face) w.r.t. the natural filtration of $(\K,w)$. We emphasize that the total order $<_l$ uniquely determines the label of faces as either positive or negative.  The above discussion can be neatly converted to an algorithm to generate birth and death times of the persistence diagram with respect to a given linear filtration of the weight function $w$.

\begin{algorithm}[ht!]
\caption{Incremental Persistence Algorithm}
\label{alg:incremental}
\begin{algorithmic}
\STATE {\bfseries Input:} $\K, w$
\STATE {\bfseries Main Procedure:}
\STATE  $F = \F$ (set of all faces in $\K$)
\STATE {\bfseries while} $F \neq \emptyset$
\begin{itemize}
\item remove a face $\sigma$ with minimum weight (w.r.t $<_l$) from $F$. Set $d = dim (\sigma)$.
\item if $\sigma$ is negative w.r.t. $\K(\sigma^-) ,$ i.e., if $\beta_{d-1}(\K(\sigma^-) \cup \sigma) = \beta_{d-1}(\K(\sigma^-)) - 1$ then add $w(\sigma)$ to $\mathcal{D}_{d-1}$ else add $w(\sigma)$ to $\mathcal{B}_d$.
\end{itemize}
\STATE {\bfseries Output} $\mathcal{B}_d, \mathcal{D}_d$, for all $d \geq 0$.
\end{algorithmic}
\end{algorithm}
The above algorithm is a simplification of the persistence algorithm in \cite[Fig. 5]{edelsbrunner2002topological} which also used negative and positive simplices. The simplification in our algorithm essentially lies in turning a blind eye to the information about the pairing between the birth and death times. The equivalence of negative faces with death times (and hence positive faces with birth times) was established in \cite[Fig. 9]{zomorodian2005computing}. These algorithms extended the incremental algorithm for computing Betti numbers in \cite{delfinado1993incremental}. We summarize the algorithm, especially for ease for future reference, as follows : Let $\sigma$ be a $d-$face in $\K$.
\begin{equation}
\label{eqn:death_birth_neg_pos}
w(\sigma) \in \mathcal{D}_{d-1} \Leftrightarrow \mbox{$\eqref{Defn:NegFace}$ holds or, \gt{alternatively,}
} \, \,  w(\sigma) \in \mathcal{B}_d \Leftrightarrow \mbox{$\eqref{Defn:PosFace}$ holds}
\end{equation}	
We end this subsection reiterating a remark with respect to our proofs.
\begin{remark}
\label{rem:uniqueness}
As already explained, if $\K$ is a weighted simplicial complex, there is a unique total ordering of the faces if the weight function is injective. Otherwise, it is only a partial ordering. However, this partial ordering can be extended to a total order. This correspondence between monotonic weights and total orders shall be used to simplify many of our proofs. We shall often prove many statements for weighted simplicial complexes with unique weights and appeal to this correspondence in extending the proof to general monotonic weight functions. Equivalently, one can prove results for $w_\discrete$ and then use the natural projection $\pi$ to obtain the corresponding result for monotonic weight function $w$.
\end{remark}

\subsubsection{Spanning acycles}
\label{sec:sp_acycle}

As made clear in the title, the other key object of our study is the {\em spanning acycle}, which has been already introduced in Definition \ref{def:SA}. We now discuss the definition in more detail. Apart from being more restrictive than that in \cite{hiraoka2015minimum,kalai1983enumeration}, our definition differs from that of \cite{hiraoka2015minimum} in its use of field coefficients over integer coefficients. Clearly, in the case of $d = 1$, $S$ is a {\em minimal spanning tree} on the graph $\K^1$. Strictly speaking, the above definition is that of a $d-$spanning acycle but since in most cases the dimension $d$ will be clear from the context, we shall not \dy{always} explicitly refer to the dimension $d$. Recall that for any $S \subseteq \K$, $w(S) = \sum_{\sigma \in S} w(\sigma)$ denotes the weight of $S$. Denoting the set of $d-$spanning acycles of $\K$ by $\S^d(\K)$, $S_0 \in \S^d(\K)$ is a {\em minimal spanning acycle} if
\begin{equation}
\label{defn:MSA}
w(S_0) = \min_{S \in \S^d(\K)}w(S).
\end{equation}
\remove{	\begin{remark}
\label{rem:SubsetWeights}
Let $\K$ be a simplicial complex with $S_1, S_2 \subset \F^d.$ Henceforth, unless stated otherwise, when we say $w(S_1) < w(S_2),$ we mean it in the spirit of Remark~\ref{rem:uniqueness}, i.e., the inequality is after the tie-breaking rule has been applied. Inequalities between other quantities retain the usual meaning.
\end{remark}
}
Spanning trees and more generally connectivity in the case of graphs can be extended in a multitude of ways to higher-dimensions. Betti numbers and acycles represent one possible (and indeed a very satisfying) way to generalize to higher dimensions. Another common generalization is via the notion of a {\em hypergraph}. In this context, one can define a hypergraph on a simplicial complex by  considering all the faces as hyper-edges. \red{We will not use hypergraph connectivity in this paper, but we only remark that studying the hypergraph connectivity of spanning acycles yields interesting results. }

\red{
\begin{remark}
We would like to highlight one more interpretation of the spanning acycles before continuing. As will no doubt be known or obvious to experts in the field, an alternative view of \textbf{a spanning acycle is as a basis for the space of boundaries}. Indeed, Algorithm 1 maintains a basis and, for insertion, checks whether the boundary of a simplex is in the span of the current basis or not. If it \gt{is} linearly independent, the simplex (or more accurately its \gt{weight} \remove{boundary}) \gt{is added to the list of death times, otherwise it is added to the set of birth times.} \remove{ otherwise it does not. }
\end{remark}}

\subsection{Probabilistic notions}
\label{sec:probability}

We give here a brief introduction to point processes on $\Real$. For a more detailed reading on weak convergence of point processes, we refer the reader to \cite[Chapter 3]{resnick2013extreme}. Let $\borR$ be the Borel $\sigma-$algebra of subsets in $\bR.$

\remove{ \begin{definition}
For $x \in \bR$ and $A \in \borR,$ define
\[
\delta_{x}(A) =
\begin{cases}
1 & \text{ if $ x \in A$,}\\
0 & \text{ otherwise.}
\end{cases}
\]
Then the function $m: \borR \to \mathbb{Z}_{\geq 0}$ given by
\[
m(\cdot) = \sum_{i = 1}^{\infty} \delta_{x_i} (\cdot),
\]
where $\{x_i\}$ is some countable (possibly finite) collection of points in $\bR$ such that $m(K) < \infty$ for every compact set $K
\in \borR,$ denotes a {\em point measure} on $\bR.$  A point measure is {\em simple} if $m(\{x\}) \leq 1$ for each
$x \in \bR.$
\end{definition}
}

A {\em point measure} on $\Real$ is a map from $\borR$ to the set of natural numbers, i.e., it is a Radon (locally-finite) counting measure. A point measure $m$ is represented as $m(\cdot) = \sum_{i = 1}^{\infty} \delta_{x_i} (\cdot),$ for some countable but locally-finite collection of points $\{x_i\}$ in $\bR$ and where $\delta_x(\cdot)$ denotes the delta measure at $x$. Alternatively, we define the support of the point measure $m$, denoted by $supp(m)$ as the multi-set $\{x_i\}$. A point measure is {\em simple} if $m(\{x\}) \leq 1$ for each $x \in \bR.$ Let $M_{p}(\bR)$ denote the set of all point measures on $\bR$. Also, let $\ccR$ denote the set of all continuous, non-negative functions $f: \bR \to \bR$ with compact support. For $f \in \ccR$ and $m = \sum_{i = 1}^{\infty} \delta_{x_i} \in M_{p}(\bR),$ define
\begin{equation}
\label{eqn:exp_wrt_pointm}
m(f) := \int_{\bR} f dm = \sum_{i = 1}^{\infty} f(x_i).
\end{equation}
Let $m_{n}, m \in M_{p}(\bR).$ We will say that $m_{n}$ converges {\em vaguely} to $m,$ denoted $m_{n} \overset{v}{\to} m,$ if for every $f \in \ccR,$ $ m_{n}(f) \to m(f)$. Using this notion of vague convergence, one defines the vague topology  on $M_{p}(\bR).$ That is, a subset of $M_{p}(\bR)$ is vaguely closed if it includes all its limit points w.r.t. vague convergence. The sub-base for this topology consists of open sets of the form
\[
\{m \in M_{p}(\bR) : m(f) \in (s, t)\}
\]
for $f \in \ccR, \, \, s, t \in \bR, \, \, s < t.$

A {\em point process} $\sP$ on $\bR$ is a random variable taking values in the space $(M_{p}(\bR), \M_{p}(\bR))$ where $\M_{p}(\bR)$ denotes the Borel $\sigma$-algebra generated by the vague topology. A point process is called {\em simple} if $\sP(\{x\}) \leq 1$ a.s. \gt{for all $x$, i.e., if it is} supported on simple point measures. An oft-used example of a point process is the Poisson point process.
\begin{definition}
\
Let $\mu : \bR \to [0, \infty)$ be locally integrable ($\int_{A} \mu(x) dx < \infty$ for all bounded $A \subset \bR$). A point process $\sP$ on $\bR$ is said to be a {\em Poisson point process} with intensity function $\mu$  if the following two properties hold.
\begin{enumerate}
\item For disjoint $A_1,\ldots,A_m \in \borR$, $\sP(A_1), \ldots, \sP(A_m)$ are independent.

\item For any $A \subseteq \bR,$ $\sP(A)$ is a Poisson random variable with mean $\int_{A} \mu(x) dx.$
\end{enumerate}
\end{definition}
\begin{definition}
Let $\sP_{n}, \sP$ be point processes on $\bR$, not necessarily defined on the same probability
space. We will say that $\sP_n$ converges weakly to $\sP,$ denoted $\sP_{n} \Rightarrow \sP,$ if
\[
\mathbb{E}[f(\sP_n)] \to \mathbb{E}[f(\sP)]
\]
for all continuous and bounded $f: (M_{p}(\bR), \M_p(\bR)) \to \bR.$ This is equivalent to saying
\[
\lim_{n \to \infty} \Pr\{\sP_{n} \in A\} = \Pr\{\sP \in A\}
\]
for all $A \in \M_{p}(\bR)$ such that $\Pr\{\sP \in \partial A\} = 0.$ Here $\partial A$ denotes the boundary of $A.$
\end{definition}

An alternative topology on $M_p(\bR)$ that arises naturally in computational topology is the so-called {\em bottleneck distance} $d_B.$ Note that we require a modified definition for point measures in $\bR$ rather than the more standard definition for persistence diagrams (e.g. \cite{chazal2009proximity,Edelsbrunner10}).

\begin{definition}
\label{def:db_point}
For $m_1,m_2 \in M_p(\bR)$
\[ d_B(m_1,m_2) := \inf_{\gamma} \sup_{x : supp(m_1)} |x - \gamma(x)|,\]
where the infimum is over all possible bijections $\gamma : supp(m_1) \to supp(m_2)$ between the multi-sets. If no bijection exists, set $d_B(m_1,m_2) = \infty$.
\end{definition}
Though this is not a metric in the classical sense, taking $\min\{d_B,1\}$ we obtain a metric on $M_p(\bR)$. More importantly, the topology induced by $d_B$ and $\min\{d_B,1\}$ are the same. We shall prove in Lemma \ref{lem:DB-Vague} that this topology is stronger than that of vague topology.

\section{Minimal spanning acycles}
\label{sec:MSA}
\gt{Our main goal here is to derive Theorems~\ref{thm:death_MSA} and \ref{thm:l1_stability}.} Additionally, we introduce several relevant combinatorial properties of (minimal) spanning acycles. To avoid tedium, we do not always single out the results for the case of minimal spanning tree, i.e., \gt{the $d = 1$ case}; since these are classical results, one can refer to \cite{Algorithms09, MSTWiki} for graph-theoretic (and expectedly simpler) proofs. Some of these results are direct consequences of the fact that the space of boundaries is a vector space \gt{and, hence, allows for a natural matroid to be defined}. Others, such as Kruskal's algorithm are folklore, but are included here for completeness as they do not appear in the literature for acycles.

\subsection{Basic Properties}

\gt{Our first aim here is to show} that a \gt{$d$-spanning acycle} exists if and only if \gt{$\beta_{d-1}(\K)=0.$} \gt{We establish this via a series of results.} \dy{As introduced in Definition \ref{def:SA}, a maximal acycle  is a natural analogous notion of a spanning forest, however \gt{note that} we mainly focus on a spanning acycle in the paper.} We \gt{begin by showing} that if a spanning acycle exists for a complex $\K$, then $\beta_{d-1}(\K) = 0$. \gt{If $S$ is a spanning acycle, then $\beta_{d - 1}(\K^{d - 1} \cup S) = 0$  by definition.} \remove{
By definition, a spanning acycle must have $\beta_{k-1}(\K^{d-1})=0$.} That this extends to the complete skeleton follows from the following corollary of Lemma~\ref{lem:delfinado}.

\begin{corollary}
\label{cor:dAndHighFacesUnaffectBettid-1}
Let $\K$ be a simplicial complex with $S_1 \subset S_2 \subset \F^d.$ Then, for any $j \geq d,$
\[
\beta_{d - 1}(\K^{d - 1}) \geq  \beta_{d - 1}(\K^{d - 1} \cup S_1) \geq \beta_{d - 1}(\K^{d - 1} \cup S_2) \geq \beta_{d - 1}(\K^d) = \beta_{d - 1}(\K^j) = \beta_{d - 1}(\K) \geq 0.
\]
\end{corollary}
\begin{proof}
By Lemma~\ref{lem:delfinado}, adding a $d$-simplex \gt{either increases} $\beta_d $ or  \gt{decreases} $\beta_{d-1}$. Since the inequalities only concern $\beta_{d-1}$,  $S_1 \subset S_2 \subset \F^d$ implies the first 3 inequalities. The equalities $\beta_{d - 1}(\K^d) = \beta_{d - 1}(\K^j) = \beta_{d - 1}(\K)$ follow from the property of simplicial complexes that for every simplex, all of its faces must be contained in the complex. Hence, \gt{adding} higher than $d$-dimensional simplices cannot change $\beta_{d-1}$.
\end{proof}

It remains to show that $\beta_{d-1}(\K) = 0$ implies the existence of a spanning acycle. We omit the case where \gt{$\F^{d}$} is empty as the $d-$spanning acycle is simply the empty set in this case. 
We begin by proving the following fact, which states that a positive simplex remains positive under simplex addition and a negative simplex remains negative under deletion. \dy{This is nothing but a restatement that the span of a basis is non-decreasing under the addition of elements. It however will be useful in the proof of correctness of Kruskal's algorithm.}
\begin{lemma}
\label{lem:+ve_faces}
Let $\K$ be a simplicial complex, $\sigma \in \K$ be a $d-$face, and let $S_1 \subseteq S_2 \subseteq \F^d$ be such that $\sigma \notin S_2$. If $\sigma$ is a positive face w.r.t. $\K^{d-1} \cup S_1$, then $\sigma$ is a positive face w.r.t $\K^{d-1} \cup S_2$. Conversely, if $\sigma$ is a negative face w.r.t. $\K^{d-1} \cup S_2$, then $\sigma$ is a negative face w.r.t. $\K^{d-1} \cup S_1$.
\end{lemma}
\begin{proof}
If $\sigma$ is a positive simplex w.rt. $\K\cup S_1,$ then \eqref{eqn:positive_image} implies $\partial\sigma \in \partial(C_d(\K\cup S_1))$. Since the $\partial(C_d(\K\cup S_1)) \subseteq \partial(C_d(\K\cup S_2),$ \gt{it follows from \eqref{eqn:positive_image}} that $\sigma$ is a positive simplex for $\K\cup S_1$ as well. The second statement is simply the contrapositive, since a simplex must be either positive or negative. 
\end{proof}
The second fact we need is a characterization of positive simplices -- if a set of $k$-simplices do not decrease $\beta_{k-1}$, then they are positive. 
\begin{lemma}
\label{lem:+ve_SA}
Let $S \subseteq \F^d$ be such that  $\beta_{d-1}(\K^{d-1} \cup S) = \beta_{d-1}(\K)$. Then any $\sigma \in \F^d \setminus S$ is a positive face w.r.t. $\K^{d-1} \cup S$. In particular, this holds when $S$ is a \dy{maximal} acycle.
\end{lemma}
	\begin{proof}
	From \gt{Corollary~\ref{cor:dAndHighFacesUnaffectBettid-1}}, $\beta_{d-1}(\K^{d-1} \cup S \cup \sigma) = \beta_{d-1}(\K)$ for any $\sigma \in \F^d \setminus S.$ Hence, from Lemma~\ref{lem:delfinado}, $\sigma$ is positive  w.r.t. $\K^{d-1} \cup S$. \gt{The case of maximal acycles follows from Definition~\ref{def:SA}}.
\end{proof}

Together, the above results imply that we can always find a simplex which will decrease the \gt{$(d-1)$-th} Betti number \gt{ whenever it is greater than that of the whole complex.}

\begin{lemma}
\label{lem:ExistNegFace}
Let $\K$ be a simplicial complex. For all $S \subseteq \F^d$, if $\beta_{d-1}(\K^{d-1} \cup S) = m > \beta_{d-1}(\K)$ then there exists a $\sigma \in \F^d \setminus S$ such that $\beta_{d-1}(\K^{d-1} \cup S \cup \sigma) = m - 1$.
\end{lemma}
\begin{proof}
Let $S \subseteq \F^d$ such that $\beta_{d-1}(\K^{d-1} \cup S) = m > \beta_{d-1}(\K)$ and suppose that all $\sigma \in \F^d \setminus S$ are positive faces w.r.t. $\K^{d-1} \cup S$. Let $S_1 \subseteq \F^d$ be such that $S_1 \supseteq S.$ Then, from Lemma \ref{lem:+ve_faces}, we have that any $\sigma \in \F^d \setminus S_1$ is positive w.r.t. $\K^{d - 1} \cup S_1.$ From this, we have $\beta_{d-1}(\K^{d-1} \cup S_1) = \beta_{d-1}(\K^{d-1} \cup S)$ for any $S_1 \supseteq S$. Taking $S_1 = \F^d$, we obtain the necessary contradiction that $\beta_{d - 1}(\K) = \beta_{d-1}(\K^{d-1} \cup S).$ Hence, there is a negative face $\sigma \in \F^d \setminus S$ w.r.t. $\K^{d-1} \cup S$.
\end{proof}

\red{As an acycle corresponds to a basis in matroid, every maximal acycle of $d$-faces has a constant cardinality (i.e., the rank of the space of boundaries). We shall prove this independently below.} \gt{Let $\gamma_d(\K) := f_d(\K^d) - \beta_d(\K^d)$.} \gt{Evaluating $\chi(\K^d) - \chi(\K^{d-1})$ by applying the   Euler-Poincar\'{e} formula \eqref{eqn:EP_formula} and then using the definition of $\gamma_d(\K)$ yields} 
\be \label{eq:useful_identity}
\gamma_d(\K)= \beta_{d-1}(\K^{d-1}) - \beta_{d-1}(\K^d) = \beta_{d-1}(\K^{d-1}) - \beta_{d-1}(\K),
\ee
where the latter equality follows from Corollary~\ref{cor:dAndHighFacesUnaffectBettid-1}.
Another use of Euler-Poincar\'{e} formula yields the following result.
\begin{lemma}(\cite[Proposition 2.13]{Duval16})
\label{lem:cardinality}
For a simplicial complex $\K$ and a subset $S \subseteq \F^d$ of $d-$faces, any two of the following three statements imply the third.
\begin{enumerate}
\item $ \beta_{d-1}(\K^{d-1} \cup S) = \beta_{d-1}(\K)$.
\item $\beta_{d}(\K^{d-1} \cup S) = 0$.
\item $|S| = \gamma_d(\K).$
\end{enumerate}
\end{lemma}
\begin{proof}
By applying the Euler-Poincar\'{e} formula to $(-1)^d[\chi(\K^{d-1} \cup S) - \chi(\K^d)] $ and re-arranging the terms, we derive the  identity:
\gt{\be
\label{eqn:RelBetBettiGamma}
\beta_d(\K^{d-1} \cup S) + \gamma_d(\K) - |S| - \beta_{d-1}(\K^{d-1} \cup S) + \beta_{d-1}(\K^d) = 0.
\ee}
\gt{ Separately, from Corollary~\ref{cor:dAndHighFacesUnaffectBettid-1}, we have $\beta_{d - 1}(\K^d) = \beta_{d - 1}(\K).$ From this, the desired result is easy to see.}
\end{proof}
\dy{From the above Lemma, we also have that the cardinality of every maximal acycle is $\gamma_d(\K)$. We can now prove the existence result} \gt{for spanning acycles}.  
\begin{lemma}
\label{lem:existence}
For a simplicial complex  $\K$, if $\beta_{d-1}(\K) = 0$, then  there exists a spanning acycle.
\end{lemma}
\begin{proof}
From ~\eqref{eq:useful_identity}, we obtain the identity
$$\beta_{d - 1}(\K^{d - 1}) = \gamma_{d}(\K) + \beta_{d-1}(\K).$$
  Since $\gamma_{d}(\K) \geq 0,$ it follows from Lemma~\ref{lem:ExistNegFace} that, starting with an empty set, we can inductively construct a set $S \subset \F^d$ such that $|S| = \gamma_d(\K)$ and $\beta_{d - 1}(\K^{d - 1} \cup S) = \beta_{d - 1}(\K).$  By Lemma~\ref{lem:cardinality}, it follows that $\beta_d(\K^{d- 1} \cup S) = 0$ implying that $S$ is an acycle with $|S| = \gamma_d(\K),$ as desired. If $\beta_{d - 1}(\K) = 0,$ then this $S$ is also a spanning acycle.
\end{proof}
\dy{We now provide a condition for uniqueness \gt{of} minimal spanning acycles and, towards \gt{deriving} the same, we first establish the exchange property \gt{of} spanning acycles.}
\begin{lemma}[Exchange property]
\label{lem:exchange}
Let $S \subset \F^d$ be a spanning acycle of a simplicial complex $\K$ and let $\sigma \in \F^d \setminus S.$ Then, for any $d$-face $\sigma_1 \in S$ such that $\sigma_1$ is part of a $d$-cycle containing $\sigma$, $S \cup \sigma \setminus
\sigma_1$ is also a spanning acycle.
\end{lemma}

\begin{proof}
By Lemma \ref{lem:+ve_SA}, $\sigma \in \F^d \setminus S$ is a positive face w.r.t. $\K^{d-1} \cup S$. So $\beta_d(\K \cup S \cup \sigma) = 1.$ Let $\sC$ be the $d$-cycle in $S \cup \sigma.$ Clearly, $\sigma \in \sC.$ So $\sum_{\tau \in \sC \cap S} a_{\tau}\partial \tau = - \partial \sigma$ for some collection of non-zero $\bF-$valued coefficients $\{a_{\tau}\}.$

Suppose that for $\sigma_1 \in S,$ $S \cup \sigma \setminus \sigma_1$ is not a spanning acycle. Then by Lemma \ref{lem:delfinado}, we obtain that $\beta_d(\K^{d-1} \cup S \cup {\sigma} \setminus \sigma_1) = \beta_{d-1}(\K^{d-1} \cup S \cup {\sigma} \setminus \sigma_1) = 1$. Let $\sC_1$ be the $d$-cycle in $\K^{d-1} \cup S \cup {\sigma} \setminus \sigma_1$.  Clearly, $\sC_1 \not\subset S$ as $S$ is a spanning acycle. Hence $\sigma
\in \sC_1$ and we derive that for some collection of non-zero $a'_{\tau} \in \mathbb{F}$, $\sum_{\tau \in (\sC_1 \cap
S)} a'_{\tau}\partial \tau  = -\partial \sigma$. Setting $a_{\tau} =0$ for $\tau \in \sC_1 \setminus \sC$ and similarly for
$a'_{\tau}$, we derive that
$$ \sum_{\tau \in (\sC
\cup \sC_1) \cap S} (a'_{\tau} - a_{\tau}) \partial \tau = \sum_{\tau \in (\sC_1 \cap S)} a_{\tau}\partial \tau - \sum_{\tau \in (\sC \cap S)} a_{\tau}\partial \tau  = \partial \sigma  - \partial \sigma = 0.$$
But since $S$ is a spanning acycle, the above implies that $\forall \tau \in (\sC \cup \sC_1) \cap S$, $a_{\tau} =
a'_{\tau}$ and hence $\sC_1 = \sC$. So, we have that $\sigma_1 \notin \sC$ if $S \cup \sigma \setminus \sigma_1$ is not a spanning acycle. By contraposition, we have that if $\sigma_1 \in \sC$, then $S \cup \sigma \setminus
\sigma_1$ is a spanning acycle.
\end{proof}
\begin{lemma}[Uniqueness]
\label{lem:uniqueness}
Let $\K$ be a simplicial complex weighted by $w: \K \to \Real$ which is injective on $\F^d$. If a minimal spanning acycle exists, then it must be unique.
\end{lemma}
\begin{proof}
Suppose that $S$ and $M$ are two distinct minimal spanning acycles. Let $\sigma$ be the $d$-face with least weight such that $\sigma \in S \vartriangle M$ and without loss of generality, assume $\sigma \in S$. Then there is a $d-$cycle $\sC \subset M \cup \sigma$ such that $\sigma \in \sC.$ Since $\sC \not\subset S,$  there exists a $d$-face $\sigma_1 \in M \setminus S$ that is part of a $d$-cycle containing $\sigma$. By the choice of $\sigma$, $w(\sigma_1) > w(\sigma).$ From Lemma~\ref{lem:exchange}, $M \cup \sigma \setminus \sigma_1$ is a spanning cycle. But $w(M \cup \sigma \setminus \sigma_1) < w(M)$, a contradiction.
\end{proof}
\begin{remark}
\label{rem:uniqueness_msa}
Suppose the weight function $w$ is not injective on $\F^d$ but nevertheless monotonic on $\K$. Then as discussed in Remark \ref{rem:uniqueness}, this weight function shall yield a total order on $\K$ and so on $\F^d$ as well. In such a case, the above theorem guarantees that the minimal spanning acycle is unique with respect to the chosen total order. 
\end{remark}

\subsection{Kruskal's algorithm}
\label{sec:algos}
\gt{The classical Kruskal's \cite{Kruskal56} algorithm helps find minimal spanning trees. We now discuss its generalization that will be useful for finding minimal spanning acycles.}
\remove{We now turn to generalizing  Kruskal's algorithm  for minimal spanning trees.} \red{Generally, greedy algorithms exist to output a minimal basis for matroids \cite[Chapter 19]{Welsh76} and the following result can be considered folklore. However, we make use of this repeatedly throughout the remainder of the paper and so we provide a self-contained proof.} 
%

Let $\K$ be a simplicial complex weighted by $w: \K \to \Real.$ By Lemma \ref{lem:delfinado}, every $\sigma \in \F^d$ is either positive or negative, but not both, with respect to a subcomplex $\K_1$ such that $\K_1^{d - 1} = \K^{d - 1}$ and $\sigma \notin \K_1$. Using this, we give the simplicial Kruskal's algorithm below.

\begin{algorithm}[ht!]
\caption{Simplicial Kruskal's Algorithm}
\label{alg:Kruskal}
\begin{algorithmic}
\STATE {\bfseries Input:} $d \geq 1,$ $\K, w$
\STATE {\bfseries Main Procedure:}
\STATE set $S = \emptyset$, $F = \F^d$
\STATE {\bfseries while} $F \neq \emptyset$ and $\beta_{d-1}(\K^{d-1} \cup S) \neq 0$
\begin{itemize}
\item remove a face $\sigma$ with minimum weight from $F$ (w.r.t. $<_l$).
\item if $\sigma$ is negative w.r.t. $\K^{d - 1} \cup S,$ then add $\sigma$ to $S$.
\end{itemize}
\STATE {\bfseries Output} $M = S$.
\end{algorithmic}
\end{algorithm}
\begin{lemma}
\label{lem:SimplicialKruskalOutcome}
Let $\K$ be a weighted simplicial complex with $\beta_{d-1}(\K) = 0$  and let $M$ be the output of the simplicial Kruskal's algorithm. Then, $M$ is a minimal spanning acycle.
\end{lemma}
\remove{In the case when weights of $\F^d$ are distinct, the Kruskal's algorithm generates the unique minimal spanning acycle but in other cases, it will generate a minimal spanning acycle depending on the tie-break rule used (see Remark \ref{rem:uniqueness}).}
%
\begin{proof}
We shall assume that the weight function $w$ is injective. For the general case, similar arguments can be carried out by using Remarks  \ref{rem:uniqueness} and \ref{rem:uniqueness_msa}. From Lemmas \ref{lem:existence} and \ref{lem:uniqueness}, it follows that there is a unique minimal spanning acycle which we denote by $M_1.$

We now show that $M$ is a spanning acycle. Clearly, $\beta_{d}(\K^{d-1}) = 0$ and by our algorithm and Lemma \ref{lem:delfinado}, it remains the same at every stage of the algorithm and so $\beta_{d}(\K^{d-1} \cup M) = 0$, proving that $M$ is an acycle. Clearly, each face in $\F^d \setminus M$ is positive with respect to $\K^{d-1} \cup M$. Hence, $M$ is spanning as using Lemma \ref{lem:+ve_faces} we have that
$$\beta_{d-1}(\K^{d-1} \cup M) = \beta_{d-1}(\K^{d-1} \cup M \cup \F^d \setminus M) = \beta_{d-1}(\K) = 0.$$

For the proof of minimality, we argue as in the Kruskal's algorithm for minimal spanning tree. We prove that at any stage of the algorithm, $S \subseteq M_1.$ Assuming that the above claim is true, $M \subseteq M_1$. Since $M$ and $M_1$ are both spanning acycles, $M_1 =  M$ as desired.

It remains to prove that $S \subseteq M_1$ at any stage. We use induction for the same. Trivially, this is true for $S = \emptyset$. Suppose that the claim holds for $S$ at some stage of the algorithm, i.e., $S \subset M_1$ but $S \neq M.$ This implies that there does exist a  $d-$face in $\F^d \setminus S$ which is negative w.r.t. $\K^{d - 1} \cup S$ and hence, from Lemma~\ref{lem:+ve_SA}, $S \neq M_1.$ Let $\sigma$ be the next face that is added to $S$ and suppose that $\sigma \notin M_1.$ Clearly, $\beta_d(\K^{d - 1} \cup M_1 \cup \sigma) = 1.$ Hence, there exists a $d-$cycle in $\K^{d - 1} \cup M_1 \cup \sigma$ whose support\footnote{For a $d-$chain $\sum_i a_i \sigma_i$, its support is $\{\sigma_i \in \F^d : a_i \neq 0\}$} $\sC$  contains $\sigma.$ Since $\beta_d(\K^{d - 1} \cup S \cup \sigma) = 0,$ $\sC \not\subseteq S \cup \sigma$ and so there exists $\sigma_1 \in \sC \cap M_1 \setminus S.$ Clearly, either $w(\sigma) < w(\sigma_1)$ or $w(\sigma) > w(\sigma_1)$ as $w$ is injective. Suppose that $w(\sigma) < w(\sigma_1).$ By the exchange property of matroids, it follows that $M_1 \cup \sigma \setminus \sigma_1$ is spanning acycle with $w(M_1 \cup \sigma \setminus \sigma_1) < w(M_1),$ a contradiction. Suppose that $w(\sigma) > w(\sigma_1).$ Since $S \subsetneq M_1,$ $\sigma_1 \in M_1 \setminus S,$ and $M_1$ is a spanning acycle, it follows from Lemma \ref{lem:+ve_faces} that $\sigma_1$ is negative w.r.t. $\K^{d - 1} \cup S.$ Thus, it follows that the algorithm would have chosen $\sigma_1$ before $\sigma,$ a contradiction. The desired claim now follows. 
\end{proof}

As with minimal spanning trees, the Kruskal's algorithm has a number of useful consequences. We conclude this section with a definition of a (topological) notion of a {\em cut} for a simplicial complex and show that it has the desired properties which will prove \dy{useful in Section~\ref{sec:uniformly_Weighted_Complex}.}
\begin{definition}[Cut]
Let $d \geq 1.$ Given a simplicial complex $\K$ with $\beta_{d-1}(\K) = 0$, a subset $\sC \subseteq \F^d$ is a cut if $\beta_{d-1}(\K - \sC) > 0$ and for any $\sC_1 \subsetneq \sC$, $\beta_{d-1}(\K - \sC_1) < \beta_{d-1}(\K - \sC)$.
\end{definition}
As expected, this definition yields a corresponding \emph{cut property}. 
\begin{lemma}[Cut Property]
\label{lem:cut}
Let $\K$ be a weighted simplicial complex with $\beta_{d-1}(\K) =0.$ Let $\sC \subseteq \F^d$ be a cut. Then $\sC \cap S \neq \emptyset$ for any spanning acycle $S$ and every minimum weight face in $\sC$ belongs to some minimal spanning acycle.
\end{lemma}

\begin{proof}
\gt{Let $S$ be a spanning acycle and suppose that $\sC \cap S = \emptyset.$ On one hand, because $S$ is spanning, $\beta_{d - 1}(\K^{d - 1} \cup S) = 0.$ On the other hand, since $\sC$ is a cut, we have $\beta_{d - 1}(\K - \sC) > 0.$ The latter, when combined with the second inequality in Corollary~\ref{cor:dAndHighFacesUnaffectBettid-1} and the fact that $S \subseteq \F^{d}(\K - \sC),$ implies $\beta_{d - 1}(\K^{d - 1} \cup S) > 0.$ This leads to a contradiction and, thus, the first conclusion holds.} 

\remove{The first conclusion follows from the definition of a cut and Lemma \ref{lem:ExistNegFace}.}

Now for the second part. Let $\sigma_1$ be a minimum weight face in the cut $\sC$ and let $<_l$ be a total order in which this is the unique minimum weight face in the cut $\sC$. Consider the simplicial Kruskal's algorithm under this $<_l$ and let $S_1$ be the acycle constructed when $\sigma_1$ is the minimum weight face in $F$. Clearly $\K^{d-1} \cup S_1 \subseteq \K - \sC$. Setting $\sC_1 = \sC \setminus \sigma_1$, the cut property implies that
$$ \beta_{d-1}(\K - \sC_1) < \beta_{d-1}(\K - \sC).$$
Thus $\sigma_1$ is negative w.r.t. $\K - \sC$ and, by Lemma \ref{lem:+ve_faces},
is also negative w.r.t. $\K^{d-1} \cup S_1.$ Hence $\sigma_1$ will be added to the minimal spanning acycle by the simplicial Kruskal's algorithm .
\end{proof}

We note that this agrees with the graph notion of a cut. This will prove useful when considering extremal faces. We conclude with the following consequence. 
Let $\K$ be a simplicial complex and let $\tau \in \F^{d - 1}.$ Then $\sigma \in \F^d$ is said to be a coface of $\tau,$ if $\tau \subset \sigma.$ Since the set of all cofaces of a $(d - 1)$-face forms a cut, the below result is immediate.
\begin{corollary}
\label{cor:nng_MSA}
Let $\K$ be a weighted simplicial complex with $\beta_{d-1}(\K) = 0$. \remove{Then for any spanning acycle $S$, $supp(\partial S)$ is the set of all $(d-1)$-faces with a $d-$coface. Further, let} \gt{Let} $\tau \in \F^{d-1}(\K)$ and $\sigma_0 := \arg\min \{ w(\sigma) : \sigma \in \F^d(\K), \tau \subset \sigma. \}$. Then, $\sigma_0 \in M$ for some minimal spanning acycle $M$.
\end{corollary}

\subsection{Persistence diagrams and minimal spanning acycles}
\label{sec:MSA_persistence}
In this section, we prove the connection between persistence diagrams and minimal spanning acycles (Theorem \ref{thm:death_MSA}) and some consequences. \dy{Though this correspondance is striking in its simplicity and completely consistent with the minimal spanning tree case, we will see that this has some non-trivial consequences in the study of weighted complexes.} 

The minimal spanning acycle represents the persistence boundary basis w.r.t. the sublevel set filtration induced by weights on the simplices. This is explicit from the incremental  algorithm (Algorithm \ref{alg:incremental}). From the decomposition of a filtration into a persistence diagram, it follows that a positive simplex generates a new homology class and hence forms a new cycle, while a negative simplex bounds an existing non-trivial homology class and hence is a boundary. \dy{Our proof will make this idea precise.}

\begin{proof}[Proof of Theorem \ref{thm:death_MSA}]
\gt{We only prove the result for death times  $\mathcal{D}$ since the result for birth times $\mathcal{B}$ is then immediate. This is because, on one hand, every $d-$simplex is either positive or negative with respect to $\K(\sigma^-)$ (see \eqref{Defn:NegFace} and \eqref{Defn:PosFace}). On the other hand, by the incremental algorithm (Algorithm \ref{alg:incremental}), negative simplices correspond to death times and positive simplices correspond to birth times \eqref{eqn:death_birth_neg_pos}. \remove{, the result for $\mathcal{D}$ implies that for $\mathcal{B}$.}}

We again only consider the case when the filtration values are unique and appeal to Remark \ref{rem:uniqueness} to complete the proof in the general case. Note that, in the general case, we use the same total ordering for the incremental algorithm (Algorithm \ref{alg:incremental}) generating death and birth times as well as the simplicial Kruskal's algorithm (Algorithm \ref{alg:Kruskal}).

By uniqueness of weights on $\F^d$, the Kruskal's algorithm gives us the minimal spanning acycle $M$. Firstly, by the relation \eqref{eqn:negative_image}, the condition to add $\sigma$ to $S$ in Kruskal's algorithm is equivalent to $\partial(C_d(S)) \subsetneq \partial(C_d(S \cup \sigma))$. \gt{Similarly,  the  incremental algorithm adds $c = w(\sigma)$ to $\mathcal{D}$ if $ \partial(C_d(\K(c-))) \subsetneq \partial(C_d(\K(c-) \cup \sigma)),$ where, for $c \in \Real,$ $\K(c) := \{\sigma \in K: w(\sigma) \leq c\}$ and  $\K(c-) := \{\sigma \in K: w(\sigma) < c\}.$} 

Let $c$ be a \gt{non-trivial} value in the filtration, i.e., there exists $\sigma \in \K$ such that $w(\sigma) = c$. Let $M(c)$ denote the acycle generated by Kruskal's algorithm on $\K(c),$ i.e., $M(c) := M \cap \K(c);$ similarly, define the notation $M(c-).$ By the above discussion on Kruskal's algorithm and incremental algorithm, our proof is complete if we show that $\partial(C_d(M(c))) = \partial(C_d(\K(c)))$. Trivially, $\partial(C_d(M(c))) \subseteq \partial(C_d(\K(c)))$ and we shall now show the other inclusion.

Suppose the other inclusion does not hold, then there exists a $\tau \in \F^d(\K(c)) \setminus M(c)$ such that $\partial \tau \notin \partial (C_d(M(c)))$. \gt{Let $w(\tau) = b \leq c$}. Then, clearly $\partial \tau \notin \partial (C_d(M(b-)));$ hence, by \eqref{eqn:negative_image}, $\tau$ will be a negative face with respect to $\K^{d-1} \cup M(b-)$. Therefore, Kruskal's algorithm would have added $\sigma$ to the acycle $M(b-)$ contradicting the assumption that $\tau \notin M(c)$. Thus, we have $\partial(C_d(M(c))) = \partial(C_d(\K(c)))$ and the proof is complete.
\end{proof}

\dy{The above result has powerful applications for random complexes as will be seen in the next section but we will now mention few applications in the deterministic setting as well.} As already mentioned in the introduction, we obtain \cite[Theorem 1.1]{hiraoka2015minimum} (see \eqref{eqn:lifetime_Hiraoka}) as an easy corollary of our previous theorem. Further, we can easily prove a fundamental uniqueness result for minimal spanning acycles relying upon this correspondence and the uniqueness of persistence diagrams \cite[Theorem 2.1]{zomorodian2005computing},\cite[Theorem 1.3]{chazal2012structure}, \cite[Theorem 1.1]{crawley2015decomposition}\footnote{Uniqueness follows from certain assumptions on finiteness and the Krull-Remak-Schmidt theorem of isomorphisms of indecomposable subgroups, which always hold in the setting of finite simplicial complexes.}.
%
%
\begin{lemma}
\label{thm:minimal spanning acycle_weights}
Let $\K$ be a weighted $d-$complex such that $\beta_{d-1}(\K^d) = 0$ and $M_1,M_2$ be two $d-$minimal spanning acycles in $\K$.
Let \gt{$c \in \Real$.}
Then we have that
\[ |\{ \sigma \in M_1 : w(\sigma) = c \}| = |\{\sigma \in M_2 : w(\sigma) = c \}| .\]
\end{lemma}
In the case of unique weights, the minimal spanning acycle is unique making the above lemma trivially true. In the case of non-unique weights, the minimal spanning acycle we obtain will depend on our choice of extension to a total order. However, the above theorem states that the weights of a minimal spanning acycle will be independent of this choice.

We now give an alternative characterization of a minimal spanning acycle that follows from the proof of Theorem \ref{thm:death_MSA}. Such a characterization of a minimal spanning tree has been very useful in the study of minimal spanning trees on infinite graphs (\cite[Chapter 11]{Lyons16}, \cite[Proposition 2.1]{Alexander1995}). A similar characterization for minimal spanning tree is known as {\em the creek-crossing criterion} in \cite{Alexander1995}.  However, we wish to point out now that these different characterizations do not coincide even in the infinite graph case (\cite[Proposition 2.1]{Alexander1995}).
\begin{lemma}
\label{lem:char_MSA}
Let $\K$ be a weighted simplicial complex with $\beta_{d-1}(\K) = 0$. Let $\sigma \in \F^d$ and $M$ be the minimal spanning acycle with respect to a total order $<_l$ extending the partial order induced by $w$. Then $\sigma \in M$ iff $\partial \sigma \notin \partial (C_d(\K(\sigma^-))).$
\end{lemma}
\begin{proof}
From the proof of Theorem \ref{thm:death_MSA}, we know that $\partial (C_d(M \cap \K(\sigma^-)) = \partial (C_d(\K(\sigma^-)).$ Thus, by Kruskal's algorithm and \eqref{eqn:negative_image}, we have that $\sigma \in M$ iff $\partial \sigma \notin \partial (C_d(\K(\sigma^-)))$.
\end{proof}	
%

\subsection{Stability result}
Here, we provide a proof for Theorem~\ref{thm:l1_stability}.
\begin{proof}[Proof of Theorem \ref{thm:l1_stability}]
\label{sec:stability}
Again, it suffices to prove the theorem for death times and the proof for birth times is quite identical. Secondly, due to Theorem \ref{thm:death_MSA}, we shall prove the stability result for weights of a minimal spanning acycle. We shall also assume $0 \leq p < \infty$ and the extension to $p = \infty$ follows by a standard limiting argument.

Let $M,M'$ be the two minimal spanning acycles corresponding to $f,f'$. We begin with the following case: where $f,f'$ differ precisely on one simplex $\sigma$ and $f(\sigma) = a, f'(\sigma)= a', |a-a'| = c$. In this case, \gt{as we shall show later}, $|M \triangle M'| \leq 2$, where $\triangle$ denotes the symmetric difference between the two sets. Since $M,M'$ have equal cardinalities, $|M \triangle M'| \in \{0,2\}.$ If $M \triangle M' = \emptyset$, we are done since the identity map between the simplices in $M,M'$ gives that
\[ \inf_{\pi} \sum_{\sigma \in M} |f(\sigma) - f'(\pi(\sigma))|^p \leq c^p = \sum_{\sigma \in \F^d}|f(\sigma)-f'(\sigma)|^p.\]
In the other case, $M \triangle M' = \{\sigma_1,\sigma_2\}$ with $\sigma_1 \in M,\sigma_2 \in M'$ and one of the $\sigma_i$'s is $\sigma$. \gt{Below}, we shall also show that $|f(\sigma_1) - f'(\sigma_2)| \leq c$. This again shows that
\[ \inf_{\pi} \sum_{\sigma \in M} |f(\sigma) - f'(\pi(\sigma))|^p \leq c^p = \sum_{\sigma \in \F^d}|f(\sigma)-f'(\sigma)|^p.\]
By a recursive application of the above case, we can prove the theorem for the general case of $f,f'$ differing in many simplices.

\gt{For the rest of the proof, we shall focus only on the case of $f,f'$ differing on exactly one simplex, say $\sigma \in \F^d,$ and derive the claims made above.} Without loss of generality, assume that $f,f'$ assign distinct weights to distinct faces; the case of non-distinct weights can be proved by appealing again to Remarks \ref{rem:uniqueness} and \ref{rem:uniqueness_msa}. \gt{Given a set $I \subset \bR$, $M(I) := \{\sigma \in M : f(\sigma) \in I\}.$ Also, as before, let $M(a) = M((-\infty,a])$ and $M(a-) = M((-\infty,a))$. \remove{Lastly, recall that $\K(a):= \{\sigma \in \K: f(\sigma) \leq a\}.$} Define these notions, similarly, for  notions for $M'.$ \remove{and $\K'.$}}

We shall break the proof into four cases where the first two take care of the trivial cases, i.e., when $M \triangle M' = \emptyset$. We shall assume that both $M$ and $M'$ are generated by simplicial Kruskal's algorithm (Algorithm \ref{alg:Kruskal}). \\[-2ex]

\noindent{\bf Case 1:} Suppose $\sigma \in M$ and $a > a'$, i.e., $f(\sigma)>f'(\sigma)$. In this case, since $M(a)$ and $M'(a)$ are both maximal acycles in $\K(a),$ we have that $|M(a)| = |M'(a)|.$ Further, by Kruskal's algorithm, $M(a'-) = M'(a'-)$. \gt{Since $\sigma \in M,$ $\sigma$ is negative w.r.t. $M(a-).$ Now, because $M'(a'-) \subset M(a-),$ it follows from Lemma \ref{lem:+ve_faces} that $\sigma$ is negative w.r.t. $M'(a'-)$ and, therefore, $\sigma \in M'.$} \remove{$\K^{d-1} \cup(a'-) \cup \sigma$ is acyclic and hence $\sigma \in M'$.} \remove{Further, since $\K^{d-1} \cup M(a-) \cup \sigma$ is acyclic,} \gt{Similarly, by Lemma~\ref{lem:+ve_faces}, it is also easy to see that $M'((a', a)) \subset M((a',a)).$ Consequently, it follows that $M(a) = M'(a)$ since $M(a)$ and $M'(a)$ have equal cardinalities.} \remove{
$M((a',a)) \subset M'$ and hence $M(a) =  M'(a)$ since $M(a),M'(a)$ have equal cardinalities.} \gt{Continuing with Kruskal's algorithm from $a$ onwards gives $M =  M'$}.\\[-2ex]

\noindent{\bf Case 2:} \gt{Suppose $\sigma \notin M$ and $a' > a$. This case is similar to Case 1 above. The main differences are as follows. First, we note that $M(a-) = M'(a-).$ Second, since $\sigma \notin M,$ $M((a, a')) \subset M'((a, a')).$ Arguing as before, it then follows that $M = M'.$} \remove{In this case, $M(a'-) = M'(a'-)$ and further $\K^{d-1} \cup M'(a'-) \cup \sigma$ is cyclic by Lemma \ref{lem:+ve_SA} and so again by Kruskal's algorithm $M = M'$.} \\[-2ex]

\noindent{\bf Case 3:} Suppose $\sigma \in M$ and $a < a'$. If $\sigma \in M'$, then arguing as in Case 1 gives $M = M'$. Thus, let $\sigma \notin M'$. \gt{We show that, for some $d-$face $\tau,$} $M'((a,a')) \setminus M((a,a')) = \{\tau\},$ $M \triangle M' = \{\sigma,\tau\},$ and $f'(\tau) - f(\sigma) \leq a' - a = c,$ as needed.

To show the same, note that by Kruskal's algorithm $M(a-) = M'(a-)$ and $|M(a')| = |M'(a')|$. \gt{Further, if $\tau'$ with $f(\tau') \in (a, a')$ is negative w.r.t. $\K(\tau'-),$ then Lemma~\ref{lem:+ve_faces} shows that $\tau'$ is also negative w.r.t. $\K'(\tau'-)$ as well. Hence, from Lemma~\ref{lem:char_MSA} and \eqref{eqn:negative_image}, it follows that $M((a, a')) \subset M'((a, a')).$ Now, because of the equality of cardinalities, there exists a $\tau \notin M$ with \gt{$f'(\tau) \in (a, a')$} \remove{$w(\tau) \in (a,a')$} such that $M'(a') = M(a) \cup M((a,a')) \cup \tau.$ The desired results are then easy to see.} \\[-2ex]

\noindent{\bf Case 4:} Suppose $\sigma \notin M$ and $a' < a$. Then either $\sigma \notin M'$ or $\sigma \in M'$. If $\sigma \notin M'$, then $M = M'$ as in Case 2. If $\sigma \in M'$, arguing as in Case 3, we have that $M((a',a)) \setminus M'((a',a)) = \{\tau\}, M(a,\infty) = M'(a,\infty)$ and hence $M \triangle M' = \{\sigma,\tau\}$ with $f(\tau) - f'(\sigma) \leq a - a' = c$. 
\end{proof}

\section{Weighted random complexes}
\label{sec:weighted_random}

\gt{Our first aim here is to look at weighted random complexes (Definition~\ref{defn:Generically_Perturbed_Weighted}) and derive our point process convergence result (Theorem~\ref{thm:perturbed_process_convergence}).
Our second aim is to show the other important consequence of our stability result (Corollary~\ref{thm:lifetimesum}).}

\gt{Towards proving Theorem~\ref{thm:perturbed_process_convergence}, we first consider a special case where the weights are i.i.d. uniform on all possible $d-$faces and $0$ elsewhere.}

\subsection{Random $d-$ complex : I.I.D. uniform weights}
\label{sec:uniformly_Weighted_Complex}

%

%
%
The {\em uniformly weighted $d-$complex } $\U_{n, d}$ is the randomly weighted $d-$ complex $\Lp_{n, d}$ with $\|\epsilon_n\|_{\infty} = 0$ and $\sF$ being the uniform distribution on \gt{$[0, 1]$} (see Definition \ref{defn:Generically_Perturbed_Weighted}); \gt{hence, $\phi = \phi'$ in this case}. The canonical filtration associated with $\U_{n, d}$ is $\{\U_{n, d}(t) : t \in [0, 1]\}.$ Trivially, the well-known random $d-$complex $Y_{n,d}(t)$ defined before Lemma \ref{lem:BettiNumber} is the same as $\U_{n,d}(t)$ in distribution.
%



Fix $d \geq 1.$ In this section, we show that the three point sets - nearest neighbour distances, death times, weights in the minimal spanning acycle - corresponding to $\U_{n,d}$ (see below Definition \ref{defn:Generically_Perturbed_Weighted}), under appropriate scaling converge to a Poisson point process as $n \to \infty$.


\subsubsection{Extremal nearest neighbour distances}
\label{sec:nnd}

Fix $\sigma \in \F^{d - 1}(\U_{n, d}).$ Then, $C(\sigma)$ defined w.r.t. $\phi,$ as in \eqref{eqn:Connection Time}, denotes the nearest neighbour distance of $\sigma.$ By considering the filtration $\{\U_{n, d}(t) : t \in [0, 1]\},$ note that $\sigma$ is isolated (not part of any $d-$face) exactly between times $0$ and $C(\sigma)$ in $\{\U_{n, d}(t) : t \in [0,1]\}.$ That is, the first coface of $\sigma$ appears at $t = C(\sigma).$ 

For each $\sigma \in \F^{d - 1}(\U_{n, d}),$ let $\bar{C}(\sigma) := n C(\sigma) - d \log n + \log(d!)$ and let $\PoiF$  be the scaled point set given by
\begin{equation}
\label{eqn:Scaled_NN_Dist_Process}
\PoiF := \{\bar{C}(\sigma): \sigma \in \F^{d - 1}(\U_{n, d})\}.
\end{equation}
Viewing the latter as a point process, for any $R \subseteq \mathbb{R},$ we set
\begin{equation}
\label{eqn:PoiFinS}
\PoiF(R) := |\{\sigma \in \F^{d - 1}(\U_{n, d}): \bar{C}(\sigma) \in R\}|.
\end{equation}
For any $c \in \Real,$ let $\PoiF(c, \infty) \equiv \PoiF((c, \infty)).$ Separately, let $N_{n, d - 1}(p)$ denote the number of isolated $(d - 1)-$faces in $Y_{n, d}(p).$ 

Since $\U_{n, d}(p)$ has the same distribution as $Y_{n, d}(p),$ it follows that $\PoiF(np - d\log n + \log(d!),\infty)$ has the same distribution as $N_{n, d - 1}(p).$  Also, whenever $p_n$ is of the form as in \eqref{eqn:pn}, then we know from Lemma \ref{lem:BettiNumber} that, as $n \to \infty,$ $N_{n, d - 1}(p_n)$ converges to  $\Poi(e^{-c}),$ the poisson random variable with mean $e^{-c}.$ From this, we have $\PoiF(c, \infty) \Rightarrow \Poi(e^{-c})$ as $n \to \infty.$ We now extend this to a multivariate convergence, thereby proving convergence of point processes $\PoiF$. Recall that  $\PoiP$ is the Poisson point process as in Theorem~\ref{thm:perturbed_process_convergence}.
\begin{prop}
\label{thm:FacesPoissonProcessConvergence}
\gt{As $n \to \infty,$} \remove{with $p_n$ as in \eqref{eqn:pn},} $\PoiF$ converges in distribution to $\PoiP$.
\end{prop}

\begin{proof} 	Let $I := \cup_{j = 1}^{m} (a_{2j-1}, a_{2j}] \subseteq \mathbb{R}$ be an arbitrary but fixed union of finite number of disjoint intervals. Since $\PoiP$ is simple and does not contain atoms, as per Lemma~\ref{lem:KallenbergResult}, it suffices to prove the following two statements in order to prove weak convergence of the point process $\PoiF$ :
\[ (i) \lim_{n \to \infty} \mathbb{E}[\PoiF(I)] = \mathbb{E}[\PoiP(I)] \, \, \, \mbox{and} \, \, \,		
(ii) \, \PoiF(I) \stackrel{d}{\Rightarrow} \PoiP(I) \, \, \mbox{as} \, \, n \to \infty. \]
\gt{In turn, to establish these two statements, we make use of the method of factorial moments, i.e., show that}
\begin{equation}
\label{eqn:fact_mom_convergence}
\Exp{[(\PoiF(I))^{(\ell)}]} \to \left(\int_{I} e^{-x} dx
\right)^{\ell} = \Exp{[(\PoiP(I))^{(\ell)}]}, \quad \forall \ell \geq 1,
\end{equation}
\gt{where, for $m \in \mathbb{N},$ the notation $m^{(\ell)} = m (m - 1) \cdots (m - \ell + 1)$ so that $\Exp{[(\PoiF(I))^{(\ell)}]}$ represents the $\ell-$th factorial moment of the random variable $\PoiF(I).$ This suffices since Statement (i) above is precisely the $\ell = 1$ case, while Statement (ii) follows due to \cite[Theorem 2.4]{Hofstad16}.} For a brief motivation on the method of factorial moments, see Appendix~\ref{app:Method.Of.Factorial.Moments}.

The rest of the proof concerns proving \eqref{eqn:fact_mom_convergence}. Let $\ell \geq 1$ be fixed. Denote $\ell-$th factorial moment of $\PoiF(I)$ by $M_{n, d}^{(\ell)}$. For $\sigma \in \F^{d - 1}(\U_{n, d})$ and $R \subseteq \Real,$ let $1(\sigma; R) \equiv \1[\bar{C}(\sigma) \in R],$ where $\1$ denotes the indicator function. Then, clearly,
\[
\PoiF(I) = \sum_{\sigma \in \F^{d - 1}(\U_{n, d})} 1(\sigma; I).
\]

\gt{Note that if $X = 1_a + 1_b,$ i.e., it is a sum of two indicators, then $X^{(2)} = 1_a 1_b + 1_b 1_a,$ while $X^{(\ell)} = 0$ for all $\ell \geq 3.$ On the other hand, if $X = 1_a + 1_b + 1_c,$ then $X^{(2)} = 2\times 1_a1_b + 2 \times 1_a1_c + 2 \times 1_b 1_c,$  $X^{(3)} = 6 \times 1_a 1_b 1_c,$ while $X^{(\ell)} = 0$ for all $\ell \geq 4.$} Proceeding along these lines, it follows using induction on $\ell$ \dy{and linearity of expectation} that 
\[
M_{n, d}^{(\ell)} = \sum_{\pmb{\sigma} \in \sI_{n, d}^{(\ell)}} \mathbb{E}\left[\prod_{i = 1}^{\ell} 1(\sigma_i; I) \right],
\]
where
\[
\sI_{n, d}^{(\ell)} :=  \{\pmb{\sigma} \equiv (\sigma_1, \ldots, \sigma_{\ell}) : \sigma_i \in \F^{d - 1}(\U_{n, d}) \text{ and
no two of } \sigma_1, \ldots, \sigma_\ell \text{ are same}\}.
\]
To simplify the computation of $M_{n, d}^{(\ell)},$ we group the faces $\pmb{\sigma} \in \sI_{n, d}^{(\ell)}$ which give the same value for $\mathbb{E}\left[\prod_{i = 1}^{\ell} 1(\sigma_i; I)\right].$ We do this as follows. For $\pmb{\sigma} \in \sI_{n, d}^{(\ell)},$ let
\[
\gamma(\pmb{\sigma}) \equiv (|\cap_{i \in S}\sigma_i|: S \subseteq
\{1, \ldots, \ell\},|S| \geq 2)
\]
denote its intersection type. For $\pmb{\sigma}, \pmb{\sigma^\prime} \in \sI_{n,
d}^{(\ell)},$ we will say that both have similar intersection type, denoted by $\pmb{\sigma} \sim
\pmb{\sigma^\prime},$ if there exists a permutation $\pi$ of the faces in $\pmb{\sigma^\prime}$ such that
$\gamma(\pmb{\sigma}) = \gamma(\pi(\pmb{\sigma^\prime})).$ It is easy to see that $\sim$ is an equivalence
relation. Let $\Gamma := \{[\pmb{\sigma}]\}$ denote the  quotient of $\sI_{n, d}^{(\ell)}$ under $\sim$ with
$[\pmb{\sigma}]$ denoting the equivalence class of $\pmb{\sigma}.$ Since the number of ways in which $\ell$ distinct $(d - 1)-$faces can intersect each other is finite, we have that the number of equivalence classes in $\Gamma,$ i.e., $|\Gamma|,$ is upper bounded by some constant (w.r.t. $n$). Indeed $|\Gamma|$ depends on $d$ and $\ell,$ but these are fixed a priori in our setup. Lastly, note that for $\pmb{\sigma} \in \sI_{n, d}^{(\ell)},$ the cardinality of its equivalence class $|[\pmb{\sigma}]|$ indeed
depends on $n.$

Fix $\pmb{\sigma} \equiv (\sigma_1, \ldots, \sigma_{\ell})$ and $\pmb{\sigma^\prime} \equiv (\sigma^\prime_1,
\ldots, \sigma^{\prime}_{\ell})$ in $\sI_{n, d}^{(\ell)}$ such that $\pmb{\sigma} \sim \pmb{\sigma^\prime}.$ Then
\[
\mathbb{E}\left[\prod_{i = 1}^{\ell} 1(\sigma_i; I)\right] = \mathbb{E}\left[\prod_{i = 1}^{\ell} 1(\sigma^\prime_i; I)\right].
\]
Hence, $M_{n, d}^{(\ell)}$ can be rewritten as
\begin{equation*}
M_{n, d}^{(\ell)}  =  \sum_{[\pmb{\sigma}] \in \Gamma} \;  \sum_{\pmb{\sigma^\prime} \in \sI_{n, d}^{(\ell)} : \,
\pmb{\sigma^{\prime}} \sim \pmb{\sigma}} \mathbb{E}\left[\prod_{i = 1}^{\ell} 1(\sigma^\prime_i; I)\right]
=  \sum_{[\pmb{\sigma}] \in \Gamma} \; |[\pmb{\sigma}]| \mathbb{E}\left[\prod_{i = 1}^{\ell} 1(\sigma_i; I)\right].
\end{equation*}
Counting the number of ways in which $\ell$ distinct $(d - 1)-$faces from a total of $\tbinom{n}{d}$ can be arranged, we have $|\sI_{n, d}^{(\ell)}| = \ell ! \binom{\binom{n}{d}}{\ell}.$ For each $[\pmb{\sigma}] \in \Gamma,$ we have $[\pmb{\sigma}] \subseteq \sI_{n, d}^{(\ell)};$ hence, $|\pmb{[\sigma]}| = c_{n} ([\pmb{\sigma}]) |\sI_{n, d}^{(\ell)}|,$ for some number $c_{n}([\pmb{\sigma}]) \in [0,1].$ Clearly
\begin{equation}
\label{eqn:ConstantSum}
\sum_{[\pmb{\sigma}] \in \Gamma}c_{n}([\pmb{\sigma}]) = 1.
\end{equation}
Hence, it follows that
\[
M_{n, d}^{(\ell)} = \sum_{[\pmb{\sigma}] \in \Gamma} \; c_{n}([\pmb{\sigma}]) \ell ! \binom{\binom{n}{d}}{\ell}
\mathbb{E}\left[\prod_{i = 1}^{\ell} 1(\sigma_i; I)\right].
\]

For every $\sigma \in \F^{d - 1}(\U_{n, d}),$ we have
\begin{equation*}
1(\sigma; I) = \sum_{j = 1}^{m}1(\sigma; (a_{2j-1}, a_{2j}]) = \sum_{j = 1}^{m}( 1(\sigma; (a_{2j-1}, \infty)) - 1(\sigma; (a_{2j}, \infty))).
\end{equation*}
Therefore, it follows that for any $\pmb{\sigma} \equiv (\sigma_1, \ldots, \sigma_{\ell}) \in \sI_{n, d}^{(\ell)},$
\[
\prod_{i = 1}^{\ell} 1(\sigma_i; I) = \sum_{(\alpha_1, \ldots, \alpha_{\ell}) \in \{1, \ldots, 2m\}^\ell}  \prod_{i = 1}^{\ell} (-1)^{\alpha_i+1}1(\sigma_i; (a_{\alpha_i}, \infty)),
\]
where $\{1, \ldots, 2m\}^\ell$ is the the $\ell-$ary cartesian power of $\{1, \ldots, 2m\}.$ Hence,
\begin{equation}
\label{eqn:Mnd_SumForm}
M_{n, d}^{(\ell)} = \sum_{[\pmb{\sigma}] \in \Gamma} \; c_{n}([\pmb{\sigma}]) \hspace{-0.5em} \sum_{(\alpha_1, \ldots, \alpha_{\ell}) \in \{1, \ldots, 2m\}^\ell }  \ell ! \binom{\binom{n}{d}}{\ell}(-1)^{\sum_i\alpha_i+ \ell}
\mathbb{E}\left[\prod_{i = 1}^{\ell} 1(\sigma_i; (a_{\alpha_i}, \infty))\right].
\end{equation}

From the scaling of $C(\sigma)$, for any $\sigma \in \F^{d -
1}(\U_{n, d})$ and any $a \in \mathbb{R},$
\[
1(\sigma;(a, \infty)) = 		\1\left[C(\sigma) > \frac{a + d \log n - \log(d!)}{n}\right].
\]
Combining this with the definitions of $C(\sigma)$ and $\U_{n,d}$, observe that
\[
\ell ! \binom{\binom{n}{d}}{\ell} \mathbb{E}\left[\prod_{i = 1}^{\ell} 1(\sigma_i; (a_{\alpha_i},
\infty))\right] \sim \frac{n^{d \ell}}{\left(d!\right)^\ell} \prod_{i = 1}^{\ell} \left(1 - \frac{a_{\alpha_i} + d \log n - \log(d!)}{n} \right)^{n - \kappa_i}.
\]
Here $\kappa_1, \ldots, \kappa_\ell \geq 0$ are some constants depending on how many vertices are common between the faces $\sigma_1, \ldots, \sigma_\ell.$ From this, irrespective of $\kappa_1, \ldots, \kappa_{\ell}$ (as these are constants independent of $n$), we have
\[
\lim_{n \to \infty} \ell ! \binom{\binom{n}{d}}{\ell} \mathbb{E}\left[\prod_{i = 1}^{\ell} 1(\sigma_i; (a_{\alpha_i},
\infty))\right]
=  e^{-\sum_{i = 1}^{\ell} a_{\alpha_i}}.
\]
\gt{
Therefore, the inner sum in \eqref{eqn:Mnd_SumForm} converges to $ \left(\sum_{j = 1}^{m} [e^{-a_{2j-1}} - e^{-a_{2j}}]\right)^{\ell} = \left(\int_{I} e^{-x} dx
\right)^{\ell}$ for every $[\pmb{\sigma}] \in \Gamma.$ Now, using \eqref{eqn:ConstantSum}, it follows that 
\[
\lim_{n \to \infty} M_{n, d}^{(\ell)} =  \left(\int_{I} e^{-x} dx
\right)^{\ell},	\]
as desired in \eqref{eqn:fact_mom_convergence}.}
\end{proof}

\subsubsection{Extremal death times}
\label{sec:death_times}

We now discuss death times in the persistence diagram. First, we state a lemma explaining why nearest neighbour distances approximate death times.
\begin{lemma}
\label{lem:ExpBettiIsoFaceZero}
Fix $d \geq 1.$ Let $N_{d - 1}(Y_{n, d}(p_n))$ be the number of isolated $(d - 1)-$faces in $Y_{n, d}(p_n)$ with $p_n$ as in \eqref{eqn:pn}. Then
\[
\lim_{n \to \infty} \EP|\beta_{d - 1}(Y_{n, d}(p_n)) - N_{d - 1}(Y_{n, d}(p_n))| = 0.
\]
\end{lemma}

This lemma essentially follows from ideas in the proofs in \cite[Theorem 1.10]{kahle2014inside}. But, to the best of our knowledge, it has not been explicitly mentioned anywhere. The proof for the case $d \geq 2$ requires cohomological arguments and hence the entire proof along with more details on cohomology theory has been provided \gt{in Section \ref{app:cohomology} in the Appendix}.


Let $\PoiD$ denote the set of scaled death times in $\Hg_{d - 1}(\U_{n, d})$ as in the second item listed below  \eqref{eqn:Connection Time}. Let $c \in \Real$ be some arbitrary but fixed constant and let $p_n$ be as defined in \eqref{eqn:pn}. Then, for $n$ large enough, we have
$$\PoiD(c, \infty) = \beta_{d - 1}(\U_{n, d}(p_{n})) \quad \text{ and } \quad \PoiF(c, \infty) = N_{d - 1}(\U_{n, d}(p_{n})). $$
From Lemma~\ref{lem:ExpBettiIsoFaceZero}, it then immediately follows that 
\begin{equation}
\label{eqn:ExpDeathBettiIsoFaceZero}
    \lim_{n \to \infty}\mathbb{E}|\PoiD(c, \infty) - \PoiF(c, \infty)| = 0.
\end{equation}
	
Now we are ready to prove the convergence result for scaled death times.
\begin{prop}
\label{thm:DeathPoissonProcessConvergence}
As $n \to \infty,$ $\PoiD$ converges in distribution to the Poisson point process $\PoiP.$
\end{prop}
\begin{proof}
Let $I := \cup_{j = 1}^{m} (a_{2j-1}, a_{2j}] \subseteq \mathbb{R}$ be some finite union of disjoint intervals. Since $\PoiP$ is simple and does not contain atoms, again as per Lemma~\ref{lem:KallenbergResult}, to prove the desired result, it suffices to show that:
\[(i)  \lim_{n \to \infty} \mathbb{E}[\PoiD(I)] = \mathbb{E}[\PoiP(I)] \, \, \, \mbox{and} \, \, \,
(ii) \, \PoiD(I) \stackrel{d}{\Rightarrow} \PoiP(I) \, \, \mbox{as} \, \, n \to \infty.\]
From triangle inequality,
\begin{equation}
\label{eqn:Diff.PoiD.PoiF.Dist}
|\PoiD(I) - \PoiF(I)|  \leq \sum_{j = 1}^{2m}|\PoiD(a_j, \infty) - \PoiF(a_j, \infty)|.		
\end{equation}
By combining this with \eqref{eqn:ExpDeathBettiIsoFaceZero} and Statement (i) from above \eqref{eqn:fact_mom_convergence}, we get (i).

The same argument also shows that $|\PoiD(I) - \PoiF(I)| \to 0$ in probability as $n \to \infty.$ Combining this with Slutsky's theorem\footnote{The relevant version of Slutsky's theorem that we use is the following: If the random variables $X,$ $X_1, X_2, \ldots,$ and $Y_1, Y_2, \ldots$ is such that $X_n \Rightarrow X$ and $|X_n - Y_n| \to 0$ in probability, then $Y_n \Rightarrow X.$} \cite[Chapter 3, Corollary 3.3]{ethier2009markov} and Statement (ii) from above \eqref{eqn:fact_mom_convergence}, we obtain (ii) as desired.
\end{proof}
\subsubsection{Extremal weights in the $d-$minimal spanning acycle}
\label{sec:minimal spanning acycle_weights}
Again fix $d \geq 1.$ Viewing $\U_{n, d}$ as a weighted simplicial complex, let $M$ denote its $d-$minimal spanning acycle. And let
\begin{equation}
\label{eqn:scaledWeights}
\PoiM := \{n w(\sigma) - d \log n + \log(d!): \sigma \in M\}
\end{equation}
denote the set of scaled weights of the faces in the $d-$minimal spanning acycle of $\U_{n, d}.$ Using Theorem~\ref{thm:death_MSA} and Proposition~\ref{thm:DeathPoissonProcessConvergence}, we get the following result
immediately.
\begin{prop}
\label{thm:WeightPoissonProcessConvergence}
As $n \to \infty,$ $\PoiM$ converges in distribution to the Poisson point process $\PoiP$.
\end{prop}

\subsection{Random $d-$ complexes : I.I.D. generic weights with perturbation}
\label{sec:extensions}

We shall now prove our most general point process convergence result (Theorem \ref{thm:perturbed_process_convergence}) and then describe corollaries which give simpler bounds to verify the assumptions of this result. For the proof, we shall first consider the simplicial complex $\K_{n}^{d}$ weighted by $\phi$ alone, which we shall refer to as $\cL_{n, d}$. With respect to this $\cL_{n, d},$ define $C(\sigma), D_i, M, \PoiF, \PoiD,$ and $\PoiM,$ exactly as below Definition \ref{defn:Generically_Perturbed_Weighted}. 
%
%
\begin{prop}
\label{thm:All_PP_Converge}
Suppose that $\sF$ is \gt{continuous}. Then, the point processes $\PoiF, \PoiD,$ and $\PoiM,$ converge in distribution to $\PoiP$ as $n \to \infty.$
\end{prop}
\begin{proof}
Clearly, $\{\sF(\phi(\sigma))\}_{\sigma \in \F^{d}(\cL_{n, d})}$ are i.i.d. uniform $[0,1]$ random variables. The desired result is now immediate from Propositions \ref{thm:FacesPoissonProcessConvergence}, \ref{thm:DeathPoissonProcessConvergence}, and \ref{thm:WeightPoissonProcessConvergence}
\end{proof}
\remove{
\begin{theorem}
\label{thm:perturbed_process_convergence}
Suppose that $\sF$ is Lipschitz continuous and strictly increasing. If $n \|\epsilon_n\|_{\infty} \to 0$ in probability, then each of $\PoiFt, \PoiDt,$ and $\PoiMt$ converges in distribution to $\PoiP.$
\end{theorem}
}
We need a comparison lemma to prove the main point process convergence result. The first inequality is obvious and the next two follow from Theorem \ref{thm:l1_stability} for $p = \infty$ and Theorem~\ref{thm:death_MSA}.
\begin{lemma}
\label{lem:max_C_Perturbation}
For fixed $n, d \geq 1,$ we have the following inequalities:
\[
\max_{\sigma \in \F^{d - 1}(\Lp_{n, d})} |C^\prime(\sigma) - C(\sigma)| \leq \|\epsilon_n\|_{\infty},
\]
\[
\inf_{\gamma} \max_{i} |\Dp_{i} - \gamma(D_{i})| \leq ||\phip - \phi||_{\infty} \leq \|\epsilon_n\|_{\infty} ,
\]
where the infimum is over all possible bijections $\gamma : \{\Dp_i\} \to \{D_i\},$ and
\[		\inf_{\gamma} \max_{i} |\phip(\sigma^\prime_i) - \gamma(\phi(\sigma_i))| \leq ||\phip - \phi||_{\infty} = \|\epsilon_n\|_{\infty},
\]
where the infimum is over all possible bijections $\gamma: \{\phip(\sigma^\prime) : \sigma^\prime \in \Mp\} \to \{\phi(\sigma) : \sigma \in M\}.$
\end{lemma}

\begin{proof}[Proof of Theorem \ref{thm:perturbed_process_convergence}.]
We only show that $\PoiDt \Rightarrow \PoiP$ as $n \to \infty$ using Lemma~\ref{lem:max_C_Perturbation}, as the other results follow similarly. Let $d_v$ be the vague metric given in \eqref{eqn:vague_metric}. Suppose we show that $d_{v}(\PoiDt, \PoiD) \to 0$ in probability as $n \to \infty.$ Then, since $(M_p(\Real),d_{v})$ is a Polish space (see Appendix \ref{sec:app}), we can apply Slutsky's theorem (\cite[Chapter 3, Corollary 3.3]{ethier2009markov}) and Proposition~\ref{thm:All_PP_Converge} to derive that $\PoiDt \Rightarrow \PoiP$ as desired. It thus suffices to prove that $d_{v}(\PoiDt, \PoiD) \to 0$ in probability as $n \to \infty.$

Let $d_B$ be as in Definition~\ref{def:db_point}. Then by Lemma \ref{lem:max_C_Perturbation}, we have that
\[
d_B(\PoiDt, \PoiD) \leq \zeta n \|\epsilon_{n}\|_{\infty},
\]
where we have assumed that the Lipschitz constant  associated with $\sF$ is $\zeta.$ Now, by assumption, $d_B(\PoiDt, \PoiD) \to 0$ in probability. Fix $\epsilon \in (0,1)$ and choose $\delta = \frac{\epsilon}{k}$ for some $k \geq 1.$ Then, we have
\begin{eqnarray}
\Pr\{d_v(\PoiDt, \PoiD) > \epsilon\} & \leq & \Pr\{d_v(\PoiDt, \PoiD) > \epsilon, d_B(\PoiDt, \PoiD)  \leq \delta\} \nonumber \\
&  & + \, \, \Pr\{d_v(\PoiDt, \PoiD) > \epsilon, d_B(\PoiDt, \PoiD) > \delta\}  \nonumber \\
& \leq & \Pr\{2\lambda_\epsilon\PoiD(K_{\epsilon})\delta > \epsilon/2\} + \, \Pr\{d_B(\PoiDt, \PoiD) > \delta\},  \nonumber \\
& \leq &  \Pr\left\{\PoiD(K_{\epsilon}) > \frac{k}{4\lambda_\epsilon}\right\} + \, \Pr\{d_B(\PoiDt, \PoiD) > \delta\}. \label{eqn:dv_convg1}
\end{eqnarray}
For the second inequality, we have used \eqref{eqn:dv_bound} with $\lambda_\epsilon$ and the compact set $K_{\epsilon}$ as given there. Since the second term in \eqref{eqn:dv_convg1} converges to $0$ as $n \to \infty$, it follows using Proposition~\ref{thm:All_PP_Converge} that
\[
\limsup_{n \to \infty} \Pr\{d_v(\PoiDt, \PoiD) > \epsilon\} \leq \lim_{n \to \infty}\Pr\left\{\PoiD(K_{\epsilon}) > \frac{k}{4\lambda_\epsilon}\right\} = \Pr\left\{\PoiP(K_{\epsilon}) > \frac{k}{4\lambda_\epsilon}\right\}.
\]
Now letting $k \to \infty$, we have that $d_v(\PoiDt, \PoiD) \to 0$ in probability as desired.
\end{proof}
Except for a few trivial cases, determining the distribution of the maximum $\|\epsilon_n\|_{\infty}$ is not easy and hence we give two simple corollaries to verify the bounds.
\begin{corollary}
\label{cor:Simple_Cond}
For each $n,$ let $\{\psi(\sigma) : \sigma \in \F^d(\Lp_{n, d})\}$ have the same distribution as the real valued random variable $\psi$ which, for some $s > 0$, satisfies $\Exp[e^{s|\psi|}] < \infty$. Define $\epsilon_n(\sigma) = a_n^{-1} \psi(\sigma)$ where $a_n$ is a sequence such that\footnote{Here $w$ is the small omega notation.} $a_n = \omega(n\log n)$. If $\sF$ is \gt{Lipschitz continuous}, then, each of $\PoiFt, \PoiDt,$ and $\PoiMt$ converges in distribution to $\PoiP$.
\end{corollary}
\begin{proof}
Using Jensen's inequality for the second inequality below and since $|\F^d(\Lp_{n, d})| \leq n^{d + 1},$
\begin{eqnarray*}
s\Exp[\|\epsilon_n\|_{\infty}] & = & \frac{1}{a_{n}} \log e^{s \Exp \left[ \max_{\sigma \in \F^{d}(\Lp_{n, d})} |\psi(\sigma)|\right]}
\leq  \frac{1}{a_{n}} \log \Exp \left[e^{s \max_{\sigma \in \F^{d}(\Lp_{n, d})} |\psi(\sigma)|]}\right]\\
& = & \frac{1}{a_{n}} \log \Exp \left[\max_{\sigma \in \F^{d}(\Lp_{n, d})} e^{s |\psi(\sigma)|}\right] \leq \frac{1}{a_{n}} \log( n^{d + 1} \mathbb{E}[e^{s|\psi|}]).
\end{eqnarray*}
Hence, $n ||\epsilon_{n}||_{\infty} \to 0$ in probability as $n \to \infty.$ The result now follows from Theorem~\ref{thm:perturbed_process_convergence}.
\end{proof}
The following corollary follows from Theorem~\ref{thm:perturbed_process_convergence} using Markov's inequality and $\|\epsilon_n\|_{\infty}  \leq \|\epsilon_n\|_1$.
\begin{corollary}
\label{cor:Simple_Condition2}
For each $n,$ let $\{\epsilon_n(\sigma): \sigma \in \F^d(\Lp_{n, d})\}$ be identically distributed random variables with \gt{$\EP|\epsilon_n(\sigma)| = o(n^{- d - 2})$} for each $\sigma.$ If $\sF$ is \gt{Lipschitz continuous}, then each of $\PoiFt, \PoiDt,$ and $\PoiMt$ converges in distribution to $\PoiP$.
\end{corollary}

%

In relation to $\epsilon_n(\sigma)$'s from Definition~\ref{defn:Generically_Perturbed_Weighted}, let $\|\epsilon_n\|_p := (\sum_{\sigma \in \F^d(\Lp_{n, d})} |\epsilon_n(\sigma)|^p)^{1/p}$ for $p \in \{1,2,\ldots\}.$

\begin{proof}[Proof of Corollary \ref{thm:lifetimesum}]
Fix a $p \in \{1,2,\ldots,\}.$ Let $\pi$ be a bijection from $\{D_i\}$ to $\{D'_i\}$ achieving the infimum in Theorem \ref{thm:l1_stability}. Due to the finiteness of the complex, such a bijection exists. Now, we derive from mean-value theorem, H\"{o}lder's inequality and our stability result (Theorem \ref{thm:l1_stability})  that
\begin{eqnarray*}
|L^p_{n,d-1} - (L')^p_{n,d-1}| & \leq & p\sum_i (D_i^{p-1} + \pi(D_i)^{p-1})|D_i - \pi(D_i)|  \\
& \leq &  p(\sum_i|D_i - \pi(D_i)|^p)^{1/p}\left[(\sum_i (D'_i)^p)^{(p-1)/p} + (\sum_i D_i^p)^{(p-1)/p}\right] \\
& \leq & p \|\epsilon_n\|_p \left[(\sum_i (D'_i)^p)^{(p-1)/p} + (\sum_i D_i^p)^{(p-1)/p}\right]
\end{eqnarray*}
Now taking expectations and again using H\"{o}lder's inequality, we obtain that
\begin{eqnarray}
\Exp[|L^p_{n,d-1} - (L')^p_{n,d-1}|] & \leq & p \Exp[\|\epsilon_n\|^p_p]^{1/p} \left( \Exp[\sum_i (D'_i)^p)]^{(p-1)/p} + \Exp[\sum_i D_i^p]^{(p-1)/p}\right) \nonumber  \\
& \leq  & p \Exp[\|\epsilon_n\|^p_p]^{1/p} \left( \Exp[L^p_{n,d-1}]^{(p-1)/p} +  \Exp[(L')^p_{n,d-1}]^{(p-1)/p}  \right) \label{eqn:lifetime_bd1}
\end{eqnarray}
By the decay bounds on $\epsilon_n(\sigma)$'s, we can derive that $n^{-(d-p)}\Exp[\|\epsilon_n\|^p_p] \to 0$. From this convergence and \eqref{eqn:hino}, our proof is complete if
we show boundedness of $n^{-(d-p)} \Exp[(L')^p_{n,d-1}]$.  Now, using our stability result (Theorem \ref{thm:l1_stability}), we obtain that
\[ \Exp[(L')^p_{n,d-1}] \leq 2^{p-1} \left( \Exp[L^p_{n,d-1}]  +  \Exp[\|\epsilon_n\|^p_p] \right),\]
and thus the required boundedness follows from convergence of  $n^{-(d-p)}\Exp[\|\epsilon_n\|^p_p]$ and \eqref{eqn:hino}.	
\end{proof}

\section*{Acknowledgements}
This research was supported through the program ``Research in Pairs" by the Mathematisches Forschungsinstitut Oberwolfach in 2015. The authors would like to thank the referee for detailed reading and numerous comments leading to an improved exposition.
\appendix	
\label{sec:app}

\section{Convergence of point processes}
\label{app:pp}

We discuss here convergence of point processes under vague topology. For notations and definitions, see Subsection~\ref{sec:topology}. Firstly, it is known that $(M_p(\Real), \M_p(\Real))$ is metrizable as a complete, separable metric space (\cite[Chapter 3, Proposition 3.17]{resnick2013extreme}), i.e., it is a Polish space. In the proof of this proposition, the vague metric $d_v$ has been used which we describe next.

Let $\{G_i\}$ be the collection of open intervals in $\Real$ with rational end points and let $\{h_j\}$ be suitable piece-wise linear approximations to the indicator function of these sets. The functions $h_j$ are chosen so that they lie in $\ccR$ and are Lipschitz continuous (while it is not explicitly highlighted, from the definition of $h_j$ in the proofs of Propositions 3.11, 3.17 from \cite[Chapter 3]{resnick2013extreme}, one can check that it is Lipschitz). Then for $m_1, m_2 \in M_p(\Real),$
%
%
%
\begin{equation}
\label{eqn:vague_metric}
d_v(m_1, m_2) := \sum_{j = 1}^{\infty} \frac{1 - \exp\{- |m_1(h_j) - m_2(h_j)|\}}{2^j}
\leq \sum_{j = 1}^{\infty} \frac{ \min \{|m_1(h_j) - m_2(h_j)|, 1\}}{2^j}.
\end{equation}

The following is an oft-used result to prove weak convergence of point processes.
\begin{lemma}\cite[Proposition 3.22]{resnick2013extreme}
\label{lem:KallenbergResult}
Let $\{\sP_n\}, \sP$ be point processes on $\bR$ with $\sP$ being simple. Let $\IR$ be the collection of all finite union of intervals in $\bR.$ Suppose that for each $I \in \IR,$ with 		$\Pr\{\sP(\partial I) = 0\} = 1$, we have
\[
\lim_{n \to \infty} \Pr\{\sP_{n}(I) = 0\} = \Pr\{\sP(I) = 0\} \, \, \mbox{and} \, \,
\lim_{n \to \infty} \EP[\sP_{n}(I)] = \EP[\sP(I)] < \infty.
\]
Then $\sP_{n} \Rightarrow \sP$ in $M_{p}(\bR).$
\end{lemma}

We now prove a lemma that will be useful when combining results from computational topology (which uses bottleneck distance) and point process theory (vague topology).
\begin{lemma}
\label{lem:DB-Vague}
The topology of bottleneck distance is stronger than that of vague topology on $M_p(\bR)$. In particular, for every $\epsilon > 0,$ there exists a constant $\lambda_\epsilon > 0$ and a compact set $K_{\epsilon}$ such that, whenever $d_B(m_1, m) \leq 1/2,$ we have
\be
\label{eqn:dv_bound}
d_v(m_1,m) \leq 2\lambda_\epsilon m(K_{\epsilon}) d_B(m_1, m) + \frac{\epsilon}{2}.
\ee
\end{lemma}
\begin{proof}
We first establish \eqref{eqn:dv_bound}. Let $\epsilon > 0$ be arbitrary. For any $m, m_1 \in M_p(\bR),$ it follows from \eqref{eqn:vague_metric} that we can choose $k$ (independent of $m, m_1$) such that:%
\be
\label{eqn:vagueMetricBd}
d_v(m_1,m) \leq \sum_{j=1}^k|m_1(h_j)-m(h_j)| + \frac{\epsilon}{2}.
\ee
Let $K_j$ be the compact support of $h_j$ and $\lambda_j,$ the associated Lipschitz constant. Set $\lambda_\epsilon = \sum_{i = 1}^{k} \lambda_j$ and $K_\epsilon = \cup_{j=1}^k K_j^1,$ where $K^{\rho} := \{x \in \Real: \exists y \in K \text{ s.t. } |x - y| \leq \rho\}.$

Let $m, m_1$ be such that $\delta := 2d_B(m_1, m) \leq 1.$ Let $\gamma : supp(m) \to supp(m_1)$ be the bijection such that $\max_{x \in \supp(m)}|x -\gamma(x)| \leq \delta$. Also, let $M = supp(m), M_1 = supp(m_1)$.

By the definition of Bottleneck distance, we have that, for any compact set $K,$
\be
\label{eqn:m1_m_Bd}
m_1(K) \leq m(K^{\delta}) \leq m_1(K^{2\delta}).
\ee
Fix a $j \in \{1, \ldots, k\}.$ By the definition of $m(h_j),$
\begin{eqnarray*}
|m_1(h_j) - m(h_j)|
&  =  &  |\sum_{x \in M}h_j(x) - \sum_{x \in M_1}h_j(x)|  = |\sum_{x \in M} [h_j(x) - h_j(\gamma(x))]| \\
& \leq & \lambda_j [m(K_j) + m_1(K_j)] \delta \leq 2\lambda_j m(K_j^1)\delta \leq 2 \lambda_j m(K_\epsilon) \delta,
\end{eqnarray*}
where in the last inequality we have used \eqref{eqn:m1_m_Bd} and the fact that $\delta \leq 1$. Substituting the above relation in \eqref{eqn:vagueMetricBd}, we get
\[
d_v(m_1,m) \leq 2\lambda_\epsilon m(K_{\epsilon}) \delta + \frac{\epsilon}{2},
\]
as desired.

From this, it follows that for every $m \in M_p(\bR)$ and $\epsilon > 0,$ there exists $\rho$ (depending on $m$ and $\epsilon$) such that $d_B(m_1, m) \leq \rho$ implies $d_v(m_1, m) \leq \epsilon,$ which completes the proof.
\end{proof}

\section{Method of Factorial Moments}
\label{app:Method.Of.Factorial.Moments}
\gt{Here, we provide a brief motivation for the method of factorial moments. First, this is very closely related to the method of moments and both these methods are useful when the goal is to establish convergence in distribution. Formally, suppose a random variable $X$ is such that its distribution is completely characterised by its moments $\{\Exp[X^k]: k \geq 1\}.$ Since polynomials are dense in the class of continuous functions, note that the above statement is true for a  broad class of random variables, including the Poisson random variable. A standard result in probability theory then states that if all the moments of a sequence of random variables converge to that of $X,$ then the sequence itself converges in distribution to $X.$ Keeping this in mind, the method of moments idea to verify if $X_n \Rightarrow X$ is to check if $\Exp[X_n^k] \to \Exp[X^k]$ or not for all $k \geq 1.$ Now, since moments of a random variable are linear combinations of its factorial moments and vice versa, we can alternatively also work with factorial moments. In case of a Poisson random variable, working with latter makes a lot of sense since the resulting  expressions are much simpler than those for the corresponding moments. In particular, if $X \sim \Poi(\lambda),$ then $\Exp[X^{(\ell)}] = \lambda^\ell.$}

\section{Betti numbers and Isolated faces in $Y_{n, d}(p)$}
\label{app:cohomology}

Lemma~\ref{lem:ExpBettiIsoFaceZero} is proved here. The cases $d = 1$ and $d \geq 2$ are dealt with separately, with the latter requiring cohomological arguments.

\begin{proof}[Proof of Lemma~\ref{lem:ExpBettiIsoFaceZero} for $d = 1$]
Let $V_n$ be the vertex set of $Y_{n, 1}(p_n).$ Since there can be at most one component of size bigger than $n/2,$ for reduced $\beta_0,$
\[
|\beta_0(Y_{n, 1}(p_{n})) - N_0(Y_{n, 1}(p_n))| \leq \sum_{V \subset V_{n}, 2 \leq |V| \leq n/2} 1_{V} + 1[N_0(Y_{n, 1}(p_n)) = n],
\]
where $1_V = 1$ whenever $V$ forms a connected component in $Y_{n, 1}(p_n)$ and there is no edge between a vertex in $V$ and a vertex in $V^c.$ If $|V| = k,$ then for all sufficiently large $n,$
\[
\mathbb{E}[1_{V}]  \leq k^{k - 2} p_n^{k - 1} (1 - p_n)^{k (n - k)}.
\]
This is because, when $|V| = k,$ there are $k^{k - 2}$ possible spanning trees in $V,$ the probability of getting a particular spanning tree in $V$ is $p_n^{k - 1},$ and the probability of having no edge between $V$ and $V^c$ is $(1 - p_n)^{k (n - k)}.$ We say sufficiently large because $p_n$ may be negative for small $n$ if $c$ is negative.  Hence, for all sufficiently large $n,$
\[
\EP|\beta_0(Y_{n, 1}(p_{n})) - N_0(Y_{n, 1}(p_n))| \leq \sum_{k = 2}^{n/2} \binom{n}{k} k^{k - 2} p_n^{k - 1} (1 - p_n)^{k (n - k)} + (1 - p_n)^{\tbinom{n}{2}}.
\]
Since $(1 - p_n) \leq e^{-p_n}$ and $\binom{n}{k} \leq e^k n^k /k^k,$ for all sufficiently large $n,$
\[
\EP|\beta_0(Y_{n, 1}(p_{n})) - N_0(Y_{n, 1}(p_n))| \leq \sum_{k = 2}^{n/2} e^{k} n^{k} k^{- 2} p_n^{k - 1} e^{- p_n k (n - k)} + e^{-p_n\tbinom{n}{2}}.
\]
As $p_n = (\log n + c)/n,$ the second term decays to $0$ with $n.$

With regards to the first term,
\[
\sum_{k = 2}^{n/2} e^{k} n^{k} k^{- 2} p_n^{k - 1} e^{- p_n k (n - k)} =  \sum_{k = 2}^{n/2} e^{k} k^{- 2} p_n^{k - 1} e^{k^2p_n -kc} \leq  \sum_{k = 2}^{n/2}  e^{T_k} ,
\]
where
\[
T_k = k - 2 \log(k) + (k - 1) \log(\log n + |c|) - (k - 1) \log n + k|c| + k^2  \left[\frac{\log n + |c|}{n}\right].
\]
Fix $n.$ Treating $k$ as a continuous variable, observe that the second derivative of $T_k$ w.r.t. $k$ is strictly positive for $k \in (3, n/2).$ This shows that $T_k$ is convex in $(3, n/2)$ and hence $T_k \leq \max\{T_3, T_{n/2}\}$ for $k \in \{3, \ldots, n/2\}.$ But $T_3 > T_{n/2}$ for all sufficiently large $n.$ Hence, for all sufficiently large $n,$
\[
\sum_{k = 2}^{n/2} e^{k} n^{k} k^{- 2} p_n^{k - 1} e^{- p_n k (n - k)} \leq  e^{T_2} + \frac{n}{2}e^{T_3}.
\]
But the RHS converges to $0$ with $n.$ The desired result now follows.
\end{proof}

We now give a brief exposition about reduced cohomology (w.r.t. $\bZ_2$ for simplicity) here which is necessary for proving Lemma~\ref{lem:ExpBettiIsoFaceZero} for the case $d \geq 2.$ In one line, it can be said that cohomology is the dual theory of homology and can be derived by considering the dual of the boundary operator $\partial$.

Consider a simplicial complex $\K.$ For $d \geq 0,$ a $d-$cochain of $\K$ is a map $g:\F^d \to \bZ_2.$ Its support\footnote{This notion of support is different than that given for chains in Section \ref{sec:simplicial_homology}} is given by $\supp(g):= \{\sigma \in \F^d:  g(\sigma) = 1\}.$ Let $C^d := \{g : \F^d \to \bZ_2\}$ denote the set of all $d-$cochains and it is $\bZ_2$-vector space under natural addition and scalar multiplication operations on $C^d$.
The $d-$th coboundary operator $\cobdr_d : C^d \to C^{d + 1}$ is defined as follows :
\begin{equation}
\label{defn:Cobdr}
\cobdr_{d}(g)(\sigma) := \sum_{\tau \in \F^{d} : \tau \subset \sigma} g(\tau), \, \, \, \, g \in C^d, \sigma \in \F^{d + 1}.
\end{equation}
For $d \geq 1,$ let $B^d := \im (\cobdr_{d - 1});$ and, for $d \geq 0,$ let $Z^d := \ker(\cobdr_{d}).$ Let $B^0 := \{0, 1\},$ where $0$ and $1$ are respectively the $0-$cochains that assign $0$ and $1$ to all vertices. The elements of $B^d$ are called {\em coboundaries} while the those of $Z^d$ are called {\em cocycles}. As in homology, we have that $\cobdr_d \circ \cobdr_{d-1} = 0$ and hence we define the $d-$th cohomology group
\[
H^d := \frac{Z^d}{B^d}.
\]
It is well known that the $d-$th homology group $H_d$ is isomorphic to the $d-$th cohomology group $H^d$  and so we have that $\beta_{d}(\K) := \rank(H^d(\K)).$

We first describe upper and lower bounds for Betti numbers, which to the best of our knowledge, have not been explicitly mentioned anywhere. But they follow from the proofs in \cite{linial2006homological,meshulam2009homological,kahle2014inside}.

We first discuss upper bounds for Betti numbers. Fix an arbitrary $d \geq 0.$ For $g \in C^d,$ let $[g] = g + B^{d}$	and
\begin{equation}
\label{eqn:wtg}
w(g) = \min\{|\supp(g^\prime)| : g^\prime \in [g]\},
\end{equation}
where $|\cdot|$ denotes cardinality. Then it follows that
\[
H^d = \{[g]: g \in Z^d\} = \{[0]\} \cup \{[g]: g \in Z^d, |\supp(g)| \geq 1, w(g) = |\supp(g)|\}
\]
and hence
\begin{equation}
\label{eqn:BetaRel}
\beta_d(\K) = \rank(\{[g]: g \in Z^d, |\supp(g)| \geq 1, w(g) = |\supp(g)|\}).
\end{equation}
For $d = 0,$ call every $\A \subseteq \F^0$ connected. For $d \geq 1,$ call $\A \subseteq \F^d$ connected, if for every $\sigma_1, \sigma_2 \in \A,$ there exists a sequence $\tau_1, \ldots, \tau_i \in \A$ with $\tau_1 = \sigma_1$ and $\tau_i = \sigma_2$ such that, for each $j,$ $\tau_j$ and $\tau_{j + 1}$ share a common $(d - 1)-$face. Now fix $d \geq 1$ and consider $g \in Z^d$ such that $\supp(g)$ is not connected. Then clearly there exists $\{\A_i : \A_i \subseteq \F^d\}$ such that each $\A_i$ is non-empty and connected; $\A_i \cap \A_j = \emptyset;$  for all $\sigma_i \in \A_i$ and $\sigma_j \in \A_j,$ $\sigma_i$ and $\sigma_j$ do not have a common $(d - 1)-$face;  and $\supp(g) = \cup_{i}\A_i.$ Let $g_{\A_i} \in C^d$ be such that $\supp(g_{\A_i}) = \A_i.$ It is then easy to see that
\[
g = \sum_i g_{\A_i}.
\]
From the above relation and our assumption that $g \in Z^d,$ it necessarily follows that each $g_{\A_i} \in Z^d.$ Suppose not. Then there exists $i$ and $\sigma \in \F^{d + 1}$ such that
\[
\cobdr_d(g_{\A_i})(\sigma) = \sum_{\tau \in \F^d, \tau \subset \sigma} g_{\A_i}(\tau) = 1.
\]
Since no $\sigma_i \in \A_i$ shares a $(d - 1)-$face with any $d-$face in $\cup_{j \neq i} \A_j,$  the above necessarily implies that $\cobdr_d(g)(\sigma) = 1;$ which is a contradiction.

From the above discussion and that the fact that the rank only depends upon independent elements, we have, for each $d \geq 0,$
\[
\beta_d(\K) = \rank(\{[g]: g \in Z^d, |\supp(g)| \geq 1, w(g) = |\supp(g)|, \supp(g) \text{ is connected}\}).
\]
If we define $1_{g} = \1[g \in Z^d]$ and
\begin{equation}
\label{eqn:G_Defn}
\G_d(\K) = \{g \in C^d: |\supp(g)| \geq 1, w(g) = |\supp(g)|, \supp(g) \text{ is connected}\},
\end{equation}
then the above discussion yields the upper bound
\begin{equation}
\label{eqn:Upp_Bd_Betti_d}
\beta_{d}(\K) \leq \sum_{g \in \G_d(\K)} 1_g.
\end{equation}

We now obtain a lower bound for the different Betti numbers. As usual, let $N_d(\K)$ denote the number of isolated $d-$faces in the given simplicial complex $\K.$ When $d = 0,$ we have
\[
N_0(\K) - 1[N_0(\K) = |\F^0| ] \leq \beta_0(\K).
\]
This shows that the number of isolated vertices is a lower bound for $\beta_0(\K)$ except in one particular case when all vertices in $\K$ are isolated. For $d \geq 1,$ however, one can easily come up with several examples when the number of isolated $d-$faces exceeds $\beta_d(\K).$ From this, it follows that $\beta_{d}(\K)$ for $d \geq 1$ needs to treated a little differently.

Fix $d \geq 1.$ In contrast to the setup used for the upper bound, we will assume here that the $d-$skeleton $\K^d$ of the given simplicial complex $\K$ is complete. Consider $\tilde{N}_{d}(\K),$ which we define to be the number of disjoint isolated $d-$faces in $\K.$ We call a $d-$face disjoint isolated if it is isolated in $\K$ and none of its neighbouring $d-$faces (i.e.,$\sigma^\prime \in \F^{d}$ which share a $(d - 1)-$face with $\sigma$) are isolated. Let $\sigma _1, \ldots, \sigma_{\tilde{N}_d}$ be all the disjoint isolated $d-$faces in $\K$ and let $g_{\sigma_1}, \ldots, g_{\sigma_{\tilde{N}_d}}$ be their associated indicator $d-$cochains. We claim that
\begin{equation}
\label{eqn:rankDisIso}
\rank(\{[g_{\sigma_1}], \ldots, [g_{\sigma_{\tilde{N}_d}}]\}) = \tilde{N}_d.
\end{equation}
If $\tilde{N}_d = 1,$ then the above is obviously true. We need to verify it for $\tilde{N}_d > 1.$ Suppose not. Then there exists $I \subseteq \{1, \ldots, \tilde{N}_d\}$ such and  a $f \in C^{d - 1}$ such that
\begin{equation}
\label{eqn:Dep}
\sum_{i \in I} g_{\sigma_i} = \cobdr_{d - 1}(f) \, \, \mbox{i.e.,} \, \, \sum_{i \in I} g_{\sigma_i} \in [0].
\end{equation}
Fix an arbitrary $i \in I$ and consider $\sigma_i.$ Suppose $\sigma_i = \{v_0, \ldots, v_d\}.$ Let $v_{d + 1} \in \F^{0}$ be such that $v_{d + 1} \notin \sigma_i.$ For $j = 0$ to $d,$ let $\tau_{ij} = \{v_0, \ldots, v_{d + 1}\} \backslash \{v_j\}.$ Note that each $\tau_{ij}$ is a $d-$face and it belongs to $\F^d$ because of our assumption that $\K^d$ is complete. Clearly none of the $\tau_{ij}$'s are $\sigma_k$'s for any $i,j,k$. Thus, from \eqref{eqn:Dep}, we have that for $i,j \in \{0,\ldots,d\}$
\[
\cobdr_{d - 1}(f)(\tau_{ij}) = 0 \, \, \mbox{and} \, \,  \cobdr_{d - 1}(f)(\sigma_i) = 1.
\]
But by the property of the coboundary operator, we derive a the contradiction that
\[
\cobdr_d (\cobdr_{d-1}(f))[v_0,\ldots,v_{d+1}] = \sum_{j = 0}^{d} \cobdr_{d - 1}(f)(\tau_{ij}) + \cobdr_{d - 1}(f)(\sigma_i) = 0.
\]
Thus \eqref{eqn:rankDisIso} holds even when $\tilde{N}_d > 1.$ From \eqref{eqn:rankDisIso} and \eqref{eqn:BetaRel}, we now have the lower bound
\begin{equation}
\label{eqn:Lwr_Bd_Betti_d}
\tilde{N}_{d}(\K) \leq \beta_{d}(\K).
\end{equation}

Some preliminary results are required before we prove Lemma~\ref{lem:ExpBettiIsoFaceZero} for $d \geq 2.$ The below result can be proved using above arguments. But we give a slightly more general proof.
\begin{lemma}
Fix $d \geq 0.$ Let $\K$ be a simplicial complex with complete $d-$skeleton such that $|\F^0| \geq d+ 2.$ Let $\sigma \in \F^d$ and $g_{\sigma}$ be the associated indicator $d-$cochain. Then $g_{\sigma} \in \G_{d}(\K).$ Also $\sigma$ is isolated in $\K$ if and only if $g_{\sigma} \in Z^d.$
\end{lemma}
\begin{proof}
It is easy to see that if $g_\sigma \notin B^d,$ then the desired result follows. For $d = 0,$ since $|\F^0| \geq 2$ and $B^0 = \{0, 1\},$ it is immediate that $g_{\sigma} \notin B^0.$ Let $d \geq 1$ and suppose that $g_{\sigma} \in B^d(\K),$ i.e., there exists $f \in C^{d - 1}(\K)$ such that $\delta_{d - 1}(f) = g_{\sigma}.$ Since $|\F^0(\K)| \geq d + 2,$ it follows that there exists $v \notin \sigma.$ Construct a new simplicial complex $\K_*$ such that $\K_*^d = \K^d$ and $\{v\} \cup \sigma$ is the only $(d + 1)-$face in $\K_*.$ Clearly $C^{d - 1}(\K_*) = C^{d - 1}(\K),$ $C^d(\K_*) = C^d(\K),$ and $B^d(\K_*) = B^d(\K),$  and hence $g_{\sigma} \in B^d(\K_*)$ and $f \in C^{d - 1}(\K_*)$ with $g_{\sigma} = \delta_{d - 1}(f).$ For $\delta_d : C^d(\K_*) \to C^{d + 1}(\K_*),$ we have $\delta_d(g_{\sigma})( \{v\} \cup \sigma) = 1.$ But this is a contradiction, since by definition of boundary maps, $\delta_d(\delta_{d - 1}(f))(\{v\} \cup \sigma) = 0.$ Hence $g_{\sigma} \notin B^d(\K_*),$ which implies that $g_{\sigma} \notin B^d(\K).$ The desired result now follows.
\end{proof}

Denoting by $N_{d}(\K)$, the number of isolated $d-$faces in $\K$, we have the following consequence of the above result :
\begin{equation}
\label{eqn:Upp_Bd_Isol_Faces}
N_{d}(\K) + \sum_{g \in \G_{d}(\K), |supp(g)| >1} 1_{g} = \sum_{g \in \G_{d}(\K)} 1_{g}.
\end{equation}
\begin{lemma}
\label{lem:Support_2_cochain}
Fix $d \geq 1.$ Let $\K$ be a simplicial complex with complete $d-$skeleton such that $|\F^0| \geq 2d + 4.$ Let $\sigma, \sigma^\prime \in \F^d$ be such that $\sigma \neq \sigma^\prime$ and $\sigma$ and $\sigma^\prime$ share a common $(d - 1)-$face. If $g_{\sigma, \sigma^\prime}$ is that $g \in C^d$ which has $\supp(g) = \{\sigma, \sigma^\prime\},$ then $g_{\sigma, \sigma^\prime} \in \G_{d}(\K).$
\end{lemma}

\begin{proof}
Since $|\supp(g_{\sigma, \sigma^\prime})| = 2$ and $\supp(g_{\sigma, \sigma^\prime})$ is connected, it suffices to show that $w(g_{\sigma, \sigma^\prime}) = 2,$ where $w(\cdot)$ is as in \eqref{eqn:wtg}. By repeating the argument below \eqref{eqn:Dep}, we can derive a contradiction to $[g_{\sigma,\sigma^{\prime}}] = [0]$ and so $w(g_{\sigma, \sigma^\prime}) \neq 0.$ Suppose $w(g_{\sigma, \sigma^\prime}) = 1.$ Then there exists some indicator $d-$co-chain $g_{\hat{\sigma}}$ associated with the $d-$face $\hat{\sigma}$ such that
\[
g_{\sigma, \sigma^\prime} +  g_{\hat{\sigma}} = \cobdr_{d - 1}(f).
\]
Now either $\hat{\sigma}$ belongs to $\{\sigma, \sigma^\prime\}$ or not. In the former case, without loss of generality, assume that $\hat{\sigma} = \sigma^\prime.$ Then it follows that $g_{\sigma} = g_{\sigma, \sigma^\prime} + g_{\hat{\sigma}} = \delta_{d - 1}(f).$ Repeating the argument below \eqref{eqn:Dep} (with $\sigma_i$ there replaced by $\sigma$), we get a contradiction. Now suppose that $\hat{\sigma} \notin \{\sigma, \sigma^\prime\}.$ Since $|\F^0| \geq 2d + 4,$ there exists $v_{d + 1} \in \F^0$ such that $v_{d + 1} \notin \sigma \cup \sigma^\prime \cup \hat{\sigma}.$ again repeating the argument below \eqref{eqn:Dep}, we again get a contradiction and hence
$w(g_{\sigma, \sigma^\prime}) = 2$ as required.
\end{proof}


\begin{theorem}
\label{thm:Betti_Iso_Bd}
Fix $d \geq 1.$ Let $\K$ be a simplicial complex with complete $d-$skeleton such that $|\F^0| \geq 2d + 4.$ Let $N_{d}(\K)$ denote the number of isolated $d -$faces in $\K.$ Then
\[
|\beta_d(\K) - N_d(\K)| \leq 3\sum_{g \in \G_{d}(\K), |\supp(g)| > 1}1_g.
\]
\end{theorem}
\begin{proof}
Let $\tilde{N}_{d}(\K)$ denote the number of disjoint isolated $d-$faces in $\K.$ Then from \eqref{eqn:Lwr_Bd_Betti_d}, \eqref{eqn:Upp_Bd_Betti_d} and \eqref{eqn:Upp_Bd_Isol_Faces}, we have
\[
\tilde{N}_{d}(\K) \leq \beta_{d}(\K) \leq \sum_{g \in \G_{d}(\K)}1_g \, \, ;  \, \,
\tilde{N}_{d}(\K) \leq N_d(\K) \leq \sum_{g \in \G_{d}(\K)}1_g. \]
Combining the above two relations, we get
\begin{equation}
\label{eqn:Betti_Iso_Gap}
|\beta_d(\K) - N_d(\K)| \leq \sum_{g \in \G_{d}(\K)}1_g - \tilde{N}_d(\K) = \sum_{g \in \G_{d}(\K)}1_g - N_d(\K) + N_d(\K) -  \tilde{N}_d(\K).
\end{equation}

Let $1(\sigma) = \1[\mbox{$\sigma$ is an isolated $d-$face in $\K$}]$ and $\hat{1}(\sigma) = \1[\mbox{$\sigma$ is a disjoint isolated $d-$face in $\K$}].$ Then, we have that
\[
N_d(\K) - \tilde{N}_d(\K) = \sum_{\sigma \in \F^d} [1(\sigma) - \hat{1}(\sigma)] \leq \sum_{\sigma \in \F^d}\sum_{\sigma^\prime} 1(\sigma) 1(\sigma^\prime),
\]
where the second sum is over all neighbouring $d-$faces of $\sigma$. For a neighbouring $d-$face $\sigma^\prime$,
\[
1(\sigma) 1(\sigma^\prime)  \leq 1_{g_{\sigma, \sigma^\prime}} ,
\]
where $g_{\sigma, \sigma^\prime}$ is that $g \in C^d$ and $\supp(g) = \{\sigma, \sigma^\prime\}.$ From the above discussion, we have
\[
N_d(\K) - \tilde{N}_d(\K) \leq 2 \sum_{g \in \G_d(\K), |\supp(g)| =2} 1_{g},
\]
The factor $2$ comes because $g_{\sigma, \sigma^\prime} = g_{\sigma^\prime, \sigma}.$ Now using  \eqref{eqn:Upp_Bd_Isol_Faces} in \eqref{eqn:Betti_Iso_Gap}, the proof is complete.
\end{proof}

\begin{lemma}
\cite[(3.5), (5.1)]{kahle2014inside}
\label{lem:NoLargeCochains}
Let $d \geq 2.$ Consider the random $d-$complex $Y_{n, d}(p_n)$ with
\[
p_n = \frac{d \log n + c - \log(d!) }{n}
\]
for some fixed $c \in \bR.$ Let $N_{d - 1}(Y_{n, d}(p_n))$ be the number of isolated $(d - 1)-$faces in $Y_{n, d}(p_n).$ Also let $\G_{d - 1}(Y_{n, d}(p_n))$ be as in \eqref{eqn:G_Defn}. Then
\[
\lim_{n \to \infty} \mathbb{E}\left[\sum_{g \in \G_{d - 1}(Y_{n, d}(p_n))} 1_g  - N_{d - 1}(Y_{n, d}(p_n))\right] = \lim_{n \to \infty} \mathbb{E}\left[\sum_{g \in \G_{d - 1}(Y_{n, d}(p_n)), |\supp(g)| > 1} 1_g\right] =  0.
\]
\end{lemma}

In \cite{kahle2014inside} (see in particular Section 5 there), each $g \in \G_{d - 1}(Y_{n, d}(p_n))$ is identified by an appropriate hypergraph $H.$ $X(H)$ is the number of $d-$faces in $Y_{n, d}(p_n)$ that contain an odd number of faces of $\supp(g).$ Hence the result follows from the following inequality :
\[
\EP[1_g] = (1 - p_n)^{X(H)} \leq e^{-p_n X(H)}.
\]

\begin{proof}[Proof of Lemma \ref{lem:ExpBettiIsoFaceZero} for $d \geq 2$]
This is easy to see from Theorem~\ref{thm:Betti_Iso_Bd} and Lemma~\ref{lem:NoLargeCochains}.
\end{proof}

\bibliographystyle{amsplain}
\bibliography{Oberwolfach}
\end{document}